%% file: main.tex
\documentclass[11pt]{article}

\usepackage[pagebackref,colorlinks=true,pdfpagemode=none,urlcolor=blue,
linkcolor=blue,citecolor=blue]{hyperref}

\usepackage{amsmath,amsfonts,amssymb,amsthm}
\usepackage{mathrsfs}
\usepackage{bm}
\usepackage{color}
\usepackage{algorithm}
\usepackage{algorithmic}
\usepackage{graphicx}
\usepackage{subfigure}

\usepackage{common_style}
\usepackage{paper_style}

\begin{document}

\title{Shape identification and classification in echolocation\thanks{\footnotesize This work was supported  by the
ERC Advanced Grant Project MULTIMOD--267184.}}
\author{Habib
Ammari\thanks{\footnotesize Department of Mathematics and
Applications, Ecole Normale Sup\'erieure, 45 Rue d'Ulm, 75005
Paris, France (habib.ammari@ens.fr, tran@dma.ens.fr,
han.wang@ens.fr).}, Minh Phuong Tran\footnotemark[2], \and Han
Wang\footnotemark[2]}

\maketitle

\begin{abstract}
  The paper aims at proposing the first shape identification and classification algorithm
  in echolocation. The approach is
based on first extracting geometric features from the reflected waves
 and then matching them with precomputed ones associated with
a dictionary of targets. The construction of such frequency-dependent shape
descriptors is based on some important properties of the scattering coefficients
 and new invariants. The stability and resolution
of the proposed identification algorithm with respect to measurement noise and the limited-view aspect
are analytically and numerically
quantified.
 \end{abstract}
\bigskip

\noindent {\footnotesize Mathematics Subject Classification
(MSC2000): 35R30, 35B30.}


\begin{flushleft}
•
\end{flushleft}\section{Introduction}
\label{sec:intro}

Echolocation is a form of acoustics that uses active sonar to locate and identify targets. Many animals, such as bats and dolphins, use this method to avoid collisions, to select, identify, and attack preys, 
and to navigate by emitting sounds and then analyzing the reflected waves \cite{nature2,nature1}. Animals with the ability of echolocation rely on multiple receivers to allow a better perception of the targets’ distance and direction. By noting a difference in sound level and the delay in arrival time of the reflected sound, the animal determines the size and the location of the target. Moreover, the echolocating animal is able to learn how
to identify certain targets and discriminate them from all other
targets. As far as we know, shape identification and classification in echolocation is not well understood yet. It is the purpose of this paper to develop the first shape identification algorithm in echolocation. 
We explain how to identify
and classify a target, knowing by advance that the latter belongs
to a certain collection of shapes. The model relies on
differential imaging, {\it i.e.}, by forming an image from the
echoes due to targets at multiple frequencies, and physics-based classification.
Considering a multistatic configuration,  we first locate the target
using  a specific location search algorithm. Then, we extract frequency-dependent scattering coefficients of the target from the perturbation of the echoes using a least-squares algorithm. Computing the invariants under rigid motions and scaling from the extracted features yields shape descriptors. Finally, a target is classified by comparing its invariants with those of a
set of learned shapes at multiple frequencies.

The main application we have in mind is an improved autonomous navigation using acoustic waves.  It is challenging to equip autonomous robots with
echolocation perception and provide them, by mimicking echolocating animals, with acoustic imaging and classification capabilities \cite{nature3}. 
 
 The present paper extends our recent results on shape description and perception in electrosensing \cite{ammari_modeling_2012, ammari_target_2012}. For the conductivity equation, the concept of generalized polarization tensors \cite{ammari_boundary_2004, ammari_reconstruction_2004,  AKLZ12, yves} has been successfully used for shape identification and classfication \cite{ammari_target_2012, ammari_shape_2013, ammari_generalized_2011}. Invariants with respect to rigid transformations and scaling for the generalized polarization tensors in two and three dimensions have been derived \cite{ammari_target_2012, ammari_invariance_2012}.

 The paper is organized as follows. In Section 2 we first recall 
  the small-volume framework \cite{ammari_reconstruction_2004,capdeboscq_review_2004, CMV98, seo,VV} for imaging acoustic targets.  In  Section 3  we recover the scattering coefficients from multistatic measurements 
using a least-squares method and estimate the resolving power of the reconstruction. 
 In Section 4
  we construct shape descriptors for multiple frequencies based on 
scaling, rotation, and translation properties of
the scattering coefficients.  In
Section 5, we illustrate numerically on test examples the perfomance of the proposed classification algorithm and its stability with respect to measurement noise and the limited-view aspect. A few concluding remarks are given in the last
section.

\section{Problem formulation}
Let $D \subset \mathbb{R}^2$ be a bounded domain with Lipschitz boundary $\partial D$. We denote by
$\ve_*$ the electric permittivity and $\mu_*$ the magnetic permeability of $D$, and $\ve_0,
\mu_0$ the electric permittivity and the magnetic permeability of the background medium $\R^2
\setminus \bar D$, respectively. Then we introduce
\begin{align}
  \ve  = \left\{ \begin{array}{ll}
      \ve_0&x \in \mathbb{R}^2 \setminus \bar D, \\
      \ve_*&x \in D,
    \end{array} \right. , \ \
  \mu  = \left\{ \begin{array}{ll}
      \mu_0&x \in \mathbb{R}^2 \setminus \bar D, \\
      \mu_*&x \in D,
    \end{array} \right. \label{defepsmu}
\end{align}
and set $k = \omega\sqrt{\ve_*\mu_*}$ and $k_0 =
\omega\sqrt{\ve_0\mu_0}$, where $\omega$ is the operating
frequency.

\subsection{Helmholtz equation}
Let $\Gamma_k$ be the fundamental solution to the Helmholtz equation:
\begin{equation}
  \label{eq:Helm}
  (\Lap+k^2) \Gamma_k(x) = \delta_0
\end{equation}
subject to the Sommerfield outgoing radiation condition. In two
dimensions, $\Gamma_k$ is given by
\begin{equation}
  \Gamma_k(x) = -\frac{i}{4}H_0^{(1)}(k\abs{x})
\end{equation}
with $H_0^{(1)}$ being the Hankel function of the first kind of
order 0.

We consider the scattered wave $u$, that is,  the solution to
following equation:
\begin{equation}
  \label{eq:uUsol}
  \begin{cases}
    \Grad \cdot \displaystyle{\frac{1}{\mu}} \Grad u + \omega^2 \ve u = 0 \qquad  {\rm in}\,\, \R^2,\\
    (u-U)\,\,\, {\rm satisfies \,\, the \,\, outgoing\,\, radiation\,\,
    condition},
  \end{cases}
\end{equation}
where $U$ is a solution to $(\Lap+k_0^2)U = 0$.

\subsubsection{Layer potentials and representation of the solution}
\label{sec:layer-potent-repr}

The solution $u$ to \eqref{eq:uUsol} admits an integral
representation. For this, we introduce the single layer potential
as follows:
\begin{equation}
  \label{SDk}
  \Sglkf D \phi (x) = \int_{\partial D} {\Gamma_k(x,y)\phi(y)d\sigma(y)}, \qquad x\in
  \mathbb{R}^2,
\end{equation}
where $\phi \in L^2(\partial D)$. The following jump formula holds \cite{ammari_reconstruction_2004}:
\begin{equation}
  \label{eq:SDk_nderiv}
  \displaystyle{\left. \frac{\partial \mathcal{S}_D^{k}[\phi] }{\partial \nu} \right|_\pm} (x) =
  \left( \pm\displaystyle{\frac{1}{2}}I + (\mathcal{K}^k_D)^* \right)[\phi](x) \qquad {\rm a.e.}
  \quad x \in \partial D,
\end{equation}
where $f|_+$ and $f|_-$ denote respectively the limit on $\partial D$ from the outside and the
inside of $D$ along the normal direction, and $\p/\p \nu$ denotes the normal
derivative. Here, $\Kstar D$ is the Neumann-Poincar\'e operator defined by
\begin{equation}
  \label{eq:KDk*}
  (\mathcal{K}^k_D)^*[\phi](x) = \int_{\partial D} {\displaystyle{\frac{\partial
  \Gamma_k(x,y)}{\partial \nu_x}}\phi(y)d\sigma(y)}
\end{equation}
with $\nu_x$ being the outward normal vectors of $\partial D$ at
$x$. Then the following system admits a unique solution
$(\varphi,\psi) \in L^2(\partial D) \times L^2(\partial D)$
\cite{ammari_boundary_2004}:
\begin{equation}
  \label{eq:psiphisys}
  \begin{cases}
    \Sglkf D \varphi - \Sglknf D \psi  = U & \text{ on } \p D, \\
    \displaystyle{ \frac{1}{\mu_*} \left.{\ddn{\Sglkf D \varphi}}\right|_{-} - \frac{1}{\mu_0}
      \left.{\ddn{\Sglknf D \psi}}\right|_{+} = \frac{1}{\mu_0} \ddn {U}} & \text{ on } \p
      D,
  \end{cases}
\end{equation}
provided that $k_0^2$ is not a Dirichlet eigenvalue for $-\Delta$ on $D$. 
Moreover, the solution $u$ to \eqref{eq:uUsol} can be written as
\begin{equation}
  \label{eq:solu}
  u(x) =
  \begin{cases}
    U(x) + \Sglknf D \psi (x), \qquad x \in \mathbb{R}^2 \setminus \bar D, \\
    \Sglkf D \varphi (x), \qquad x \in D.
  \end{cases}
\end{equation}

\subsubsection{Asymptotic expansion}
\label{sec:asymptotic-expansion} We first recall Graf's formula
\cite{abramowitz_handbook_1964}:
\begin{equation}
  \label{eq:graf}
  H_0^{(1)}(k|x-y|) = \sum\limits_{n \in \mathbb{Z}}
  {H_n^{(1)}(k|x|)e^{in\theta_x}J_n(k|y|)e^{-in\theta_y} } \ \text{ for }
  |x|>|y|,
\end{equation}
where $H_n^{(1)}$ is the Hankel function of the first kind of
order $n$ defined by
\begin{align}
  \label{eq:Hankel_n_def}
  H_n^{(1)}(x) = J_n(x) + i Y_n(x),
\end{align}
and $J_n, Y_n$ are the Bessel function of the first and the second kind of order $n$, respectively.

Then by plugging Graf's formula into \eqref{eq:solu}, we obtain an
asymptotic form of $u-U$ as $|x| \to \infty$:
\begin{equation}
  \label{eq:umU}
  u(x) - U(x) = -\displaystyle{\frac{i}{4}}\sum\limits_{n \in \mathbb{Z}}
  {H_n^{(1)}(k_0|x|)e^{in\theta_x}\int_{\partial D}J_n(k_0|y|)e^{-in\theta_y}\psi(y)d\sigma(y)}.
\end{equation}
In the following, we use $\cwv_m$ to denote the cylindrical wave
of index $m\in\Z$ and of wave number $k_0$, which is defined by
\begin{align}
  \label{eq:cylind_wav}
  \cwv_m(x)=\cwv_{m,k_0}(x):=J_m(k_0|x|)e^{im\theta_{x}}.
\end{align}

\subsection{Scattering coefficients}
\subsubsection{Definition}
Let $(\varphi_m, \psi_m)$ be the solution to \eqref{eq:psiphisys} with the cylindrical wave
$\cwv_{m}$ as the source term $U$. We introduce the \emph{scattering coefficients} as follows
\cite{ammari_enhancement_2011-2}.
\begin{dfn}
  The scattering coefficients $W_{mn}$ ($m,n \in \mathbb{Z}$) associated with the permittivity and
  permeability distributions $\ve, \mu$ given by (\ref{defepsmu}) and the frequency $\omega$ are defined by
  \begin{equation}
    \label{eq:sct}
    W_{mn} = W_{mn}[D] = W_{mn}[D, \ve, \mu, \omega] := \int_{\partial D}
    {\widebar{\cwv_n(y)} \psi_m(y)d\sigma(y)}.
  \end{equation}
\end{dfn}

The coefficient $W_{mn}$ decays very fast as the orders $m,n$
increase. In fact, it has been proved in
\cite{ammari_enhancement_2011-2} that there is a constant $\CW$
depending on $(D, \ve, \mu, \omega)$ such that
\begin{align}
  \label{eq:control_SCT}
  \abs{W_{mn}[D, \ve,\mu,\omega]} \leq \frac{\CW^{\abs m + \abs n}}{\abs{m}^{\abs m} \abs{n}^{\abs
      n}}\,\, \text{ for all } m,n\in\Z.
\end{align}


\subsubsection{Transformation formulas}
\label{sec:SCT_transform_formulae} We introduce the notation for
translation, scaling and rotation of a vector $x$ as $x^z:=x+z,
x^s:=sx, x^\theta:=R_\theta x$, and those of a shape as $$D^z:=D+z,
D^s:=sD, D^\theta:=R_\theta D,$$ and define new functions $f^z,
f^\theta, f^s$ so that
$$f^z(x^z)=f^\theta(x^\theta)=f^s(x^s)=f(x).$$

Explicit relations exist between the scattering coefficients
of $D$ and $D^z, D^s, D^\theta$. We will need the following
technical lemma:
\begin{lem}
  \label{lem:psi_Dz_Uz}
  We denote by $\psif{U,D}=\psif{U,D,\omega}$ the solution to \eqref{eq:psiphisys} given the domain
  $D$, the source term $U$ and the frequency $\omega$. Then for any $z\in\R^2$, $\theta\in\R$ and
  $s>0$:
  \begin{align}
    \label{eq:psi_Dz_Uz}
    \psif{U,D}(x) = \psif{U^z, D^z}(x^z) = \psif{U^\theta, D^\theta}(x^\theta) = s \psif{U^s, D^s,
      \omega/s}(x^s), \, \forall x\in\p D.
  \end{align}
\end{lem}
\begin{proof}
  We first prove  $\psif{U,D}(x) = \psif{U^z, D^z}(x^z)$. The following equalities are easily
  verified:
  \begin{align}
    \label{eq:Sgl_Kstar_Dz}
    \Sglkf{D^z}{\phi^z}(x^z) &= \Sglkf{D}{\phi}(x) \, \text{ for } x \in \R^2 ,\\
    \Kstarf{D^z}{\phi^z}(x^z) &= \Kstarf{D}{\phi}(x) \, \text{ for } x\in \p D
    \notag,
  \end{align}
  which imply, using the jump formula \eqref{eq:SDk_nderiv}, that if $(\varphi, \psi)$ is the
  solution to \eqref{eq:psiphisys}, $(\varphi^z, \psi^z)$ will be the solution to the following
  system:
  \begin{equation}
    \label{eq:psiphisys_Dz}
    \begin{cases}
      \Sglkf{D^z}{\varphi} - \Sglknf{D^z}{\psi}  = U^z & \text{ on } \p D^z, \\
      \displaystyle{ \frac{1}{\mu_*} \left.{\ddn{\Sglkf{D^z}{\varphi}}}\right|_{-} - \frac{1}{\mu_0}
        \left.{\ddn{\Sglknf{D^z}{\psi}}}\right|_{+} = \frac{1}{\mu_0} \ddn {U^z}} & \text{ on } \p D^z.
    \end{cases}
  \end{equation}
  Hence we obtain $\psif{U^z, D^z}(x) = \psi^z(x)$ and so forth the first identity in
  \eqref{eq:psi_Dz_Uz}. The proof of the identity $\psif{U,D}(x) = \psif{U^\theta,
    D^\theta}(x^\theta)$  follows analogous arguments and is skipped for brevity.
    
  We prove now $\psif{U,D}(x) = s \psif{U^s, D^s, \omega/s}(x^s)$. Similarly, the following
  equalities can be verified:
  \begin{align*}
    s^{-1} \Sglf{D^s}{k/s}{\phi^s}(x^s) &= \Sglkf{D}{\phi}(x) \, \text{ for } x \in \R^2, \  \\
    \frac{\p \Sglf{D^s}{k/s}{\phi^s}(x^s)}{\p \nu} \bigg|_{\pm} &= \frac{\p \Sglkf{D}{\phi}(x)}{\partial \nu} \bigg|_{\pm}  \, \text{ for }
    x\in \p D,
  \end{align*}
  which implies that if $(\varphi, \psi)$ is the solution to \eqref{eq:psiphisys}, then $(\varphi^s,
  \psi^s)$ will be the solution to the following system:
  \begin{equation}
    \label{eq:psiphisys_Ds}
    \begin{cases}
      \Sglf{D^s}{k/s}{\varphi} - \Sglf{D^s}{k_0/s}{\psi}  = sU^s & \text{ on } \p D^s, \\
      \displaystyle{ \frac{1}{\mu_*} \left.{\ddn{\Sglf{D^s}{k/s}{\varphi}}}\right|_{-} -
        \frac{1}{\mu_0} \left.{\ddn{\Sglf{D^s}{k_0/s}{\psi}}}\right|_{+} = \frac{1}{\mu_0} \ddn
        {\Paren{sU^s}}} & \text{ on } \p D^s.
    \end{cases}
  \end{equation}
  Hence, we obtain $\psif{sU^s, D^s, \omega/s}(x) = \psi^s(x)$, and so forth the last identity in
  \eqref{eq:psi_Dz_Uz} by noticing that $\psif{sU^s, D^s, \omega/s}(x) = s\psif{U^s, D^s,
    \omega/s}(x)$.
\end{proof}

\begin{prop}
  \label{prop:SCT-tsr}
  For any $z\in\R^2, \theta\in[0, 2\pi), s>0$, the following relations hold:
  \begin{subequations}
    \label{eq:SCT-tsr}
    \begin{align}
      W_{mn}[D^z] &= \sum_{a,b\in\Z} {\cwv_{a}(z)} \widebar{\cwv_{b}(z)} W_{m-a,n-b}[D] ,\label{eq:SCT-trl}\\
      W_{mn}[D^\theta] &= e^{i(m-n)\theta} W_{mn}[D] ,\label{eq:SCT-rot}\\
      W_{mn}[D^s, \ve,\mu,\omega] &= W_{mn}[D, \ve,\mu,s\omega] . \label{eq:SCT-scl}
    \end{align}
  \end{subequations}
\end{prop}
\begin{proof}
  We establish only \eqref{eq:SCT-trl} here, and the identities \eqref{eq:SCT-rot} and
  \eqref{eq:SCT-scl} can be proved in a similar way using Lemma \ref{lem:psi_Dz_Uz} and change of
  variables.

  We write $\psif{\cwv_m, D}(y):=\psi_m(y)$ in order to highlight the dependence on $\cwv_m$ and
  $D$, and recall the identity \cite{han_book}:
  \begin{align}
    \label{eq:cylwav_sum}
    \cwv_{m}(x+y) = \sum_{l\in\Z} \cwv_{l+m}(x) \widebar{\cwv_l(-y)}.
  \end{align}
  Then from Lemma \ref{lem:psi_Dz_Uz} and by the principle of superposition, it follows that
  \begin{align}
    \label{eq:psi_cwv_sum}
    \psif{\cwv_m, D}(y) =  \psif{\cwv^{z}_m, D^{z}}(y^{z}) = \sum_{l\in\Z} \widebar{\cwv_l(z)}
    \psif{\cwv_{l+m}, D^{z}}(y^{z}).
  \end{align}
  By the change of variables $y\rightarrow y-z$ in \eqref{eq:sct} and by substituting \eqref{eq:psi_cwv_sum}, we get
  \begin{align*}
    W_{mn}[D] &= \int_{\p D} \widebar{\cwv_n(y)} \psif{\cwv_m, D}(y) d\sigma(y)
    = \sum_{a,b\in\Z} \widebar{\cwv_{a}(z)} {\cwv_{b}(z)} \int_{\p D^z}
    \widebar{\cwv_{b+n}(y)} \psif{\cwv_{a+m}, D^z}(y) d\sigma(y)\\
    &= \sum_{a,b\in\Z} \widebar{\cwv_{a}(z)} {\cwv_{b}(z)}
    W_{a+m,b+n}[D^z],
  \end{align*}
  which is equivalent to
  \begin{align*}
    W_{mn}[D^z] = \sum_{a,b\in\Z} \widebar{\cwv_{a}(-z)} {\cwv_{b}(-z)}
    W_{a+m,b+n}[D].
  \end{align*}
Therefore, we obtain \eqref{eq:SCT-trl} by the fact
$\cwv_a(-z)=\widebar{\cwv_{-a}(z)}$. Finally, we
  remark that the sum in \eqref{eq:SCT-tsr} converges absolutely, due to the decay
  \eqref{eq:control_SCT} of $W_{mn}$ and the decay \eqref{eq:Bessel_bound_JY} of the Bessel
  functions.
\end{proof}

Using \eqref{eq:psi_Dz_Uz} and \eqref{eq:Sgl_Kstar_Dz}, we prove easily the identity:
\begin{align*}
  \Sglf{D}{k_0}{\psif{U, D}}(x) = \Sglf{D^{z}}{k_0}{\psif{U^{z}, D^{z}}}(x^{z}) \, \text{ for } x \in
  \R^2.
\end{align*}
Consequently, by the integral representation formula
\eqref{eq:solu} this yields the following result.
\begin{cor}
  \label{prop:umU_uzmUz}
  Let $u_{\sss{U,D}}$ be the solution to~\eqref{eq:psiphisys} given the domain $D$ and the source
  term $U$. Then for any $z\in\R^2$:
  \begin{align}
    \label{eq:uDzmUz_eq_umU}
    u_{\sss{U,D}}(x) - U(x) = u_{\sss{U^z,D^z}}(x^z) - U^z(x^z), \ \forall\, x\in \R^2 \setminus \bar
    D.
  \end{align}
\end{cor}


\section{Reconstruction of scattering coefficients and stability analysis}
\label{sec:reconstr-form-scts}
In this section we investigate the reconstruction of scattering coefficients from the measurements, and
analyze the stability of the reconstruction. We consider a multistatic configuration \cite{han_book}. Formally, let $u_s$ be the solution to \eqref{eq:uUsol}
corresponding to the source term $U_s$ for $s=1,2\ldots N_s$, and $\{x_r \}$ for $r = 1,2,...,N_r$
be the set of receivers. The multistatic response matrix (MSR) $\V:=(V_{sr})_{sr}$ is defined by
\begin{equation}
  \label{eq:Vrs}
  V_{sr} :=  u_s(x_r) - U_s(x_r) \quad {\rm for} \quad r = 1\ldots N_r, \text{ and } s = 1\ldots N_s.
\end{equation}
The entry $V_{sr}$ is the measurement recorded by the $r$-th receiver when the $s$-th source $U_s$ is in use.

\subsection{Full aperture acquisition}
\label{sec:full-angle-view}

In the following, we adopt a circular acquisition system using plane waves, which refers to the
situation that the receivers $x_r$ are uniformly distributed on a circle of radius $R$ and centered
at $z_0$ (typically, $z_0$ can be obtained using some localization algorithm and we assume that
$z_0$ is close to the center of $D$), and the sources are plane waves with equally distributed wave
direction; see Figure~\ref{fig:acqsys} (a). More specifically, for the $r$-th receiver we have
$\abs{x_r-z_0}=R$ and the angle $\theta_r:=\theta_{x_r-z_0} = 2\pi r/N_r$, where as the $s$-th plane wave source is given by
\begin{align}
  \label{eq:Us}
  U_s(x)=e^{ik_0 \xi_s\cdot x}
\end{align}
with the unit vector $\xi_s$ satisfying $\theta_s:=\theta_{\xi_s}=2\pi s/N_s$. Thanks to
\eqref{eq:uDzmUz_eq_umU}, we have
\begin{align}
  \label{eq:Vsr_umU_uzmUz}
  V_{sr} = u_{\sss{U_s,D}}(x_r) - U_s(x_r) = u_{\sss{U_s^{-z_0},D^{-z_0}}}(x_r-z_0) -
  U_s^{-z_0}(x_r-z_0),
\end{align}
which means that the measurement recorded at $x_r$ given the inclusion $D$ and the source $U_s$ is the
same as that recorded at $x_r-z_0$ given the inclusion $D^{-z_0}$ and the source $U_s^{-z_0}$.

\subsection{Linear system of scattering coefficients}
\label{sec:linear-system-sct}

We deduce from \eqref{eq:Vsr_umU_uzmUz} a linear system involving $W_{mn}$.
Recall first the Jacobi-Anger decomposition of plane waves:
\begin{equation}
  \label{eq:JA_plw}
  U_s^{-z_0} (x) = e^{ik_0 \xi_s \cdot(x+z_0)} = e^{i k_0 \xi_s \cdot z_0}
  \sum\limits_{m \in \mathbb{Z}} {e^{im(\frac{\pi}{2} -
  \theta_{s})}\cwv_{m}(x)},
\end{equation}
and by the principle of superposition, the solution $\psif{U_s^{-z_0}, D^{-z_0}}$ to
\eqref{eq:psiphisys} given the inclusion $D^{-z_0}$ and the source $U_s^{-z_0}$ reads:
\begin{align}
  \label{eq:psipsim}
  \psif{U_s^{-z_0}, D^{-z_0}}(x) = e^{i k_0 \xi_s\cdot z_0} \sum_{m \in \Z} {e^{im(\frac{\pi}{2} - \theta_{s})}
    \psi_{\cwv_m, D^{-z_0}}(x)}.
\end{align}
On the other hand, we assume $R$ large enough so that $R > \abs y$
for $y\in\p D^{-z_0}$ and  Graf's formula \eqref{eq:graf} holds
with $x$ replaced by $x_r-z_0$. Then we apply the asymptotic
expansion \eqref{eq:umU} on the last term of
\eqref{eq:Vsr_umU_uzmUz}, by substituting \eqref{eq:psipsim}, to
finally obtain
\begin{align}
  \label{eq:Vrs_ex}
  V_{sr} &= -\displaystyle{\frac{i}{4}} \sum\limits_{n \in \mathbb{Z}} H_n^{(1)}(k_0 R)e^{in\theta_r}
  \sum_{m \in \Z} e^{ik_0\xi_s\cdot z_0} {e^{im(\frac{\pi}{2} - \theta_{s})} \int_{\partial
      D^{-z_0}} \widebar{\cwv_n(y)} \psi_{\cwv_m, D^{-z_0}}(y)} d\sigma(y) \notag \\
  &= -\displaystyle{\frac{i}{4}} \sum\limits_{n \in \mathbb{Z}} H_n^{(1)}(k_0 R)e^{in\theta_r}
  \sum_{m \in \Z} e^{ik_0\xi_s\cdot z_0} e^{im(\frac{\pi}{2} - \theta_{s})}
  W_{mn}[D^{-z_0}].
\end{align}
This motivates us to introduce the constants $A_{sm}$ and $B_{rn}$:
\begin{equation}
  \label{eq:AsmBrn}
  A_{sm} = e^{ik_0\xi_s\cdot z_0} e^{im(\frac{\pi}{2}-\theta_s)} ; \quad B_{rn} =
  \displaystyle{\frac{i}{4}}\widebar{H_n^{(1)}(k_0 R)}e^{-in\theta_r}, \quad m,n \in \mathbb{Z},
\end{equation}
and rewrite \eqref{eq:Vrs_ex} as
\begin{align}
  \label{eq:Vrs_ex2}
  V_{sr} = \sum_{m,n\in\Z} A_{sm} {W}_{mn}[D^{-z_0}] {(B^H)}_{nr},
\end{align}
where $B^H$ denotes the Hermitian conjugate transpose of $B$.

\subsubsection{Connection with the far-field pattern}\label{sec:far-field-pattern}

The \emph{far-field pattern} (the scattering amplitude) when the incident field is the plane wave
$U(x)=e^{i k_0\xi\cdot x}, \abs\xi=1$ at the frequency $\omega$, is a two-dimensional
$2\pi$-periodic function $\Af_D(\cdot;\omega)$ defined to be \cite{ammari_enhancement_2011-2}
\begin{align}
  \label{eq:far_field_pattern_def}
  \Af_D({(\theta_\xi, \theta_x)}^\top; \omega) = \int_{\p D} e^{-ik_0 \abs y \cos(\theta_x -
    \theta_y)} \psif{U,D}(y) d\sigma(y)
\end{align}
with $\theta_\xi,\theta_x\in[0,2\pi]$ being respectively the angle
of $\xi$ and $x$ in polar coordinates.

In particular, for the circular acquisition system mentioned in
Section \ref{sec:full-angle-view} the following relation holds
between the MSR and the far-field pattern.
\begin{prop}
For a circular acquisition system entered at $z_0$ and the radius of the measurement circle $R\to +\infty$:
  \begin{align}
    \label{eq:Vsr_far_field}
    \abs{V_{sr}} \simeq \frac{1}{\sqrt{8 \pi k_0 R}} \Abs{\Af_{D^{-z_0}}({(\theta_s, \theta_r)}^\top; \omega)}.
  \end{align}
  Moreover, there exists a constant $C=C(k_0, k, D^{-z_0})$ such that
  \begin{align}
    \label{eq:Vsr_bound}
    \abs{V_{sr}} \leq C\Paren{\frac{\abs{\p D}}{\sqrt R}} \ \text{ when } R \rightarrow +\infty.
  \end{align}
\end{prop}
\begin{proof}
  Recall the following asymptotic of the Hankel function
  \begin{align}
    \label{eq:H0_asymp_t_inf}
    H_0^{(1)}(t) = \sqrt{\frac 2{\pi t}} e^{i\Paren{t -\frac \pi 4}} + O(1/t) \ \text{ as }
    t\rightarrow \infty.
  \end{align}
  Substituting in the integral representation of \eqref{eq:Vsr_umU_uzmUz}, we find
  \begin{align}
    \label{eq:Vsr_far_field_intexpr}
    V_{sr} &= \Sglf{D^{-z_0}}{k_0}{\psif{U_s^{-z_0}, D^{-z_0}}}(x_r) \notag \\
    &\simeq - \frac i 4 e^{-\frac{\pi i} 4}e^{i k_0 R} \sqrt{\frac{2}{\pi k_0 R}} \int_{\p
      D^{-z_0}} e^{-ik_0 \abs y \cos(\theta_x-\theta_y)} \psif{U_s^{-z_0}, D^{-z_0}}(y) d\sigma(y),
  \end{align}
  which gives \eqref{eq:Vsr_far_field} by definition of the far-field pattern after taking the
  modulus.

  It is proved in \cite{ammari_boundary_2004} that there exists a constant $C=C(k, k_0, D)$
  such that the solution $(\phi, \psi)$ of \eqref{eq:psiphisys} satisfies:
  \begin{align}
    \label{eq:phipsi_energy_control}
    \norm{\varphi}_{L^2(\p D)} + \norm{\psi}_{L^2(\p D)} \leq C( \norm{U}_{L^2(\p D)} + \norm{\Grad
      U}_{L^2(\p D)}).
  \end{align}
  For the plane wave source $U_s$, this implies:
  \begin{align*}
    \norm{\psif{U_s, D}}_{L^2(\p D)} \leq C (1+k_0) \abs{\p D}^{1/2}.
  \end{align*}
  Substituting back in \eqref{eq:Vsr_far_field_intexpr} and applying Cauchy-Schwartz inequality, we
  obtain \eqref{eq:Vsr_bound}.
\end{proof}

\subsubsection{Truncated linear system}
\label{sec:trunc-line-syst}
A truncated version of \eqref{eq:Vrs_ex2} using the first $\abs m, \abs n \leq K$ terms reads
\begin{equation}
  \label{eq:V_AWB_Ers}
  V_{sr} = \sum_{m=-K}^{K}\sum_{n=-K}^{K} {A_{sm} {W}_{mn}[D^{-z_0}] {(B^H)}_{nr}} +
  E_{sr},
\end{equation}
where $E_{sr}$ denotes the truncation error. For the fixed
expansion-order $K$, we introduce the matrix $\W =
({W}_{mn}[D^{-z_0}])_{mn}$ of dimension $(2K+1)\times(2K+1)$, the
source matrix $\A=(A_{sm})_{sm}$ of dimension $N_s\times (2K+1)$
and the receiver matrix $\B=(B_{rn})_{rn}$ of dimension $N_r\times
(2K+1)$, as well as the truncation error matrix $\E=(E_{sr})_{sr}$
of dimension $N_s\times N_r$. Then \eqref{eq:V_AWB_Ers} can be put
into a matrix form
\begin{align}
  \label{eq:AWBE_mat}
  \V = \A \W \B^H + \E  = \bL(\W) + \E,
\end{align}
where the linear operator $\bL: \mathbb{R}^{(2K+1)\times(2K+1)}
\rightarrow \mathbb{R}^{N_s \times
  N_r}$ is given by
\begin{equation}
  \label{eq:opL}
  \bL(\mbf X) = \A \mbf X \B^H.
\end{equation}
In practical situations,  the data $\V$ may be contaminated by
some noise $\Wnoise$,  thereby modifing \eqref{eq:AWBE_mat} to 
\begin{equation}
  \label{eq:AWBE_mat_noise}
  \V = \bL(\W) + \E + \Wnoise .
\end{equation}
We suppose that $\Wnoise = \snoise \bN_0$, where $\snoise>0$ and $\bN_0$ is a $N_s$-by-$N_r$ complex
random matrix with independent and identically distributed $\normallaw 1$ entries.

From the bound \eqref{eq:Vsr_bound}, one can see the size of the measurement $V_{sr}$ is of order
${\abs{\p D}}/{\sqrt R}$, hence we define signal-to-noise ratio (SNR) to be
\begin{align}
  \label{eq:def_SNR}
  \SNR = \frac{\abs{\p D}/{\sqrt R}}{\snoise}.
\end{align}

\subsection{Bound of the truncation error $E_{sr}$}
\label{sec:trunc_err_E}
In this section we estimate the truncation error $E_{sr}$ which will be used later to determine the
maximal resolving order $K$ that one can achieve in a noisy environment. The following results will
be useful:
\begin{itemize}
\item For a fixed $t$ when $n$ is large, we have the following asymptotic behaviors of the Bessel functions \cite{abramowitz_handbook_1964}:
  \begin{align}
    \label{eq:Bessel_bound_JY}
    J_{n}(t) \simeq \displaystyle{\sqrt{\frac{1}{2\pi|n|}}\left(\frac{et}{2|n|} \right)^{|n|}}, \ \
    Y_{n}(t) \simeq \displaystyle{-\sqrt{\frac{2}{\pi|n|}}\left(\frac{et}{2|n|}
    \right)^{-|n|}},
  \end{align}
  which implies an upper bound of the Hankel function:
  \begin{align}
    \label{eq:Hankel_bound}
    \Abs{H_{n}^{(1)}(k_0R)}
    \lesssim {\Paren{C_R\abs n}^{\abs n} + \Paren{C_R\abs n}^{-\abs n}
    },
  \end{align}
  where $C_R:= 2/(ek_0 R)$.
\item For $c>0$ and $k\in\N$ such that $k> c/e$, the following bound holds:
  \begin{align}
    \label{eq:xpx_sum}
    \sum_{m>k}\Paren{\frac{c}{m}}^m \leq \Paren{\frac c k}^{k}
    \Paren{\frac{1}{1+\ln(k/c)}}.
  \end{align}
  The proof is given in Appendix \ref{sec:proof-of-xpx-sum}.
\end{itemize}

\begin{prop}
  \label{prop:Esr_bound_K}
  Let $C_R$ be the constant defined in \eqref{eq:Hankel_bound}, and $\CW>1$ be the constant in
  \eqref{eq:control_SCT} with the shape $D^{-z_0}$. Suppose the radius of measurement circle $R$
  satisfies $\CW^2C_R<1$.  Then there exists a sufficiently large truncation order $K$ satisfying
  $K>\CW/(C_R e)$, such that the truncation error $E_{sr}$ at the order $K$ decays as
  \begin{align}
    \label{eq:Esr_bound_K}
    \abs{E_{sr}} = O(\rho^{-K})
  \end{align}
  with $\rho=\Paren{\CW^2 C_R}^{-1}>1$ a constant depending only on $D^{-z_0},\ve, \mu, \omega,$ and $R$.
\end{prop}
\begin{proof}
  We separate $E_{sr}$ in three terms as follows:
  \begin{align*}
    {E_{sr}} = \Paren{\sum_{\substack{|m| \le K\\|n| > K}} + \sum_{\substack{|n| \le K\\|m| >
          K}}+\sum_{\substack{|m| > K\\|n| > K}}} {A_{sm} W_{mn} (B^H)_{rn}} = I_1 + I_2 +
          I_3.
  \end{align*}
  Let the radius $R$ be sufficiently large so that $\CW^2C_R <1$. Then by \eqref{eq:control_SCT} and
  \eqref{eq:Hankel_bound}, we can bound $\abs{I_1}$ by
  \begin{align*}
    \abs{I_1} &\leq \sum_{|m|\leq K} \frac{\CW^{\abs m}}{\abs{m}^{\abs m}} \sum_{|n| > K}
    \Abs{H^{(1)}_n(k_0 R)}  \frac{\CW^{\abs n}}{\abs{n}^{\abs n}}\\
    &\lesssim \Paren{\sum_{|m|\leq K} \frac{\CW^{\abs m}}{\abs{m}^{\abs m}}} \Paren{\sum_{\abs n > K}
      {\Paren{\CW C_R}^{\abs n}} + \sum_{\abs n > K} \Paren{\frac{\CW/C_R} {\abs{n}^2}}^{\abs n}}.
  \end{align*}
  Up to some factors independent of $K$, we can bound the first sum in the last expression by
  $\CW^K$, and the second sum by $(\CW C_R)^K$ (since $\CW C_R<\CW^2 C_R<1$). Furthermore, as a consequence of
  \eqref{eq:xpx_sum} the last sum can be neglected by choosing
  \begin{align*}
    K > \CW / (C_R e)
  \end{align*}
  sufficiently large. Therefore, we obtain $\abs{I_1} \lesssim (\CW^2C_R)^{K}$. Proceeding in a similar
  way one can establish for $I_2$ and $I_3$:
  \begin{align*}
    \abs{I_2}\lesssim (\CW^2C_R)^{K}, \ \text{ and }  \abs{I_3}\lesssim (\CW^2C_R)^{K} K^{-K}.
  \end{align*}
  Putting these bounds together gives \eqref{eq:Esr_bound_K}.
\end{proof}

\subsection{Stability of the linear operator L}
\label{sec:stab-line-oper}

Here we show that the operator $\bL$ is ill-conditioned for the
circular acquisition system mentioned in section
\ref{sec:linear-system-sct}. If $N_s, N_r \ge 2K+1$, then it can
be verified easily that the matrices $\A,\B$ defined in
\eqref{eq:AsmBrn} are orthogonal:
\begin{align}
  \label{eq:As_Ar_ortho}
  \A^H \A = N_s \I, \ \text{ and } \B^H \B = N_r \bD,
\end{align}
where $\I$ is the identity matrix of dimension $\Paren{2K+1} \times \Paren{2K+1}$ and the $\Paren{2K+1} \times \Paren{2K+1}$ diagonal matrix 
\begin{align}
  \label{eq:P_matrix}
  \bD =
  \begin{pmatrix}
    |d_{-K}|^2 &  0  & \ldots & 0\\
    0  &  |d_{-K+1}|^2 & \ldots & 0\\
    \vdots & \vdots & \ddots & \vdots\\
    0  &   0   &\ldots & |d_K|^2
  \end{pmatrix}\
  \ \text{ with }  d_n = \displaystyle{\frac{i}{4}}{H_n^{(1)}(k_0 R)}.
\end{align}

\begin{prop}
  \label{prop:SVD_L}
  Suppose $N_s, N_r \ge 2K+1$. The right singular vector of $\bL$ is the canonical basis of
  $\R^{(2K+1) \times (2K+1)}$, and the $(m,n)$-th singular value of operator $\bL$ is
  \begin{equation}
    \label{eq:sing_vl}
    \lambda_{mn} = \sqrt{N_sN_r}|d_n|, \ \forall\, m,n=-K,\ldots, K
    .
  \end{equation}
\end{prop}

\begin{proof}
  \label{prf:sing_vl}
  For $m,n=-K\ldots K$, let $\bv_{mn}$ be the $\Paren{2K+1} \times \Paren{2K+1}$ square matrix
  \begin{equation}
    (\bv_{mn})_{ij} = \delta_{mi}\delta_{nj}, \ \forall\, i,j=-K,\ldots, K.
  \end{equation}
  We define the inner product of two complex matrices as $\seq{A, B}:=\sum_{m,n} (A^H)_{mn}
  B_{mn}$. A simple computation using \eqref{eq:As_Ar_ortho} shows
  \begin{align*}
    \seq{\bL(\bv_{mn}),\bL(\bv_{m'n'})} = N_s N_r \seq{\bv_{mn}, \bv_{m'n'}\bD} =
    \delta_{mm'}\delta_{nn'} N_s\,N_r\,\abs{d_n}^2.
  \end{align*}
  This proves that the canonical basis $\set{\bv_{mn}}_{m,n=-K\ldots K}$ constitutes the right 
  singular vectors of $\bL$. The $(m,n)$-th singular value of $\bL$ is hence given by
  \begin{equation}
    \lambda_{mn} = ||\bL(\bv_{mn})||_F = \sqrt{N_sN_r}|d_n|, \quad m,n= -K, \ldots, K, 
  \end{equation}
  where $\norm{\cdot}_F$ denotes the Frobenius norm of matrices.  The left singular vectors are given by   
  $\bu_{mn}:= \bL(\bv_{mn})/\lambda_{mn}$.
  This completes the proof.
\end{proof}

\subsubsection{Condition number of $\bL$}
\label{sec:condition-number-bL} Remark that as $K\rightarrow
+\infty$, $\abs{d_K}$ diverges, and hence $\bL$ is an unbounded
operator and the following bound holds for the condition number of
$\bL$:
\begin{cor}
  \label{cor:cond_L}
  Under the same assumptions of Proposition \ref{prop:SVD_L}, we have
  \begin{align}
    \label{eq:condnum_L_bound}
    \mathrm{\cond}\,(\bL) \lesssim (C_R K)^{K}  \ \text{ as } K \rightarrow +\infty
  \end{align}
  with the constant $C_R$ defined in \eqref{eq:Hankel_bound}.
\end{cor}
\begin{proof}
  We denote respectively by $\dmax$ and $\dmin$ the maximum and the minimum value of $\abs{d_n}$ for
  $n=-K\ldots K$.  Using the asymptotic expansion in \eqref{eq:Bessel_bound_JY}, we obtain
  \begin{align}
    \label{eq:dn_asymptotic}
    \abs{d_n} \propto \Abs{H^{(1)}_n(k_0 R)} \simeq \sqrt{\frac {2}{\pi \abs n}} \Paren{C_R \abs
      n}^{\abs n} \ \text{ as } \abs n\rightarrow +\infty.
  \end{align}
  Hence, when $K$ is large, $\dmax\simeq d_K$ and $\dmin$ is bounded from below. Therefore, the condition
  number is bounded by
  \begin{align*}
    \cond(\bL) = \frac\dmax \dmin \lesssim (C_R K)^{K},
  \end{align*}
  where the underlying constant depends only on $\dmin$.
\end{proof}

\subsection{Least-squares reconstruction}\label{sec:least-square-reconst}
We can reconstruct $\W$ by solving the least-squares problem
\begin{align}
  \label{eq:least_square}
  \West = \argmin_{\W\in\ker\bL^\bot} \Norm{\bL(\W)-\V}_F,
\end{align}
where $\ker\bL$ denotes the kernel of $\bL$. 
The solution of (\ref{eq:least_square}) is given by $\West=\bL^\dagger(\V)$ with $\bL^\dagger$ being the pseudo-inverse of
$\bL$. Since the order $K$ satisfies $2K+1 <N_s, N_r$, the matrices $\A$ and $\B$ have full columns rank and both of them have left pseudo-inverses denoted by $\A^\dagger$ and $\B^\dagger$, respectively. Then we  have the pseudo-inverse formulas 
\begin{align}
  \label{addeq}
\A^\dagger = (\A^H \A)^{-1} \A^H,\quad \B^\dagger = (\B^H \B)^{-1} \B^H.
\end{align}
Using the orthogonal property \eqref{eq:As_Ar_ortho} of $\A$ and $\B$ together with (\ref{addeq}), we verify that
\begin{align}
  \label{eq:L_pseudoinv}
  \bL^\dagger(\V) = \frac{1}{N_s N_r} \A^H \V \B \D^{-1},
\end{align}
which provides an analytical reconstruction formula for $\W$.

\subsection{Maximal resolving order}
\label{sec:maxim-resolv-order}
Here we give an estimation of the maximal resolving order $K$ that can be reconstructed by
\eqref{eq:least_square} as a function of the SNR.

Let $\rho$ be the constant in Proposition \ref{prop:Esr_bound_K}. We work under the assumption
\begin{equation}
  \label{eq:meas_noise}
  \rho^{-K} \ll \snoise \ll {\abs{\p D}}/{\sqrt R},
\end{equation}
which is the regime that the measurement noise is much larger than the truncation error but much
smaller than the signal (or $\SNR \gg 1$).

Since $\bL$ is injective, $\bL^\dagger \bL = \I$ and from \eqref{eq:AWBE_mat_noise} we have
\begin{align*}
  \Exp{\Abs{\Paren{\West - \W}_{mn}}^2}^{1/2} = \Exp{\abs{{\bL^\dagger(\E+\Wnoise)}_{mn}}^2}^{1/2}
  \leq \abs{\bL^\dagger(\E)_{mn}} + \snoise\lambda_{mn}^{-1}\sqrt{N_sN_r}.
\end{align*}
On one hand, by Cauchy-Schwartz inequality and the bound \eqref{eq:Esr_bound_K}:
\begin{align*}
  \abs{\bL^\dagger(\E)_{mn}} \leq \lambda_{mn}^{-1} \norm{\E}_F \lesssim \lambda_{mn}^{-1}
  \sqrt{N_sN_r} \rho^{-K},
\end{align*}
which can be neglected thanks to the left inequality of
\eqref{eq:meas_noise}. Therefore, we obtain
\begin{align}
  \label{eq:recon_error_bound_1}
  \Exp{\Abs{\Paren{\West - \W}_{mn}}^2}^{1/2} \lesssim \snoise\lambda_{mn}^{-1}\sqrt{N_sN_r} =
  \snoise \abs{d_n}^{-1},
\end{align}
which decreases to zero for all $m,n\geq K$ as $K\rightarrow
+\infty$, due to \eqref{eq:dn_asymptotic}. Apparently, this seems
to imply that the maximal resolving order is infinity. However,
from the decay behavior \eqref{eq:control_SCT} of $W_{mn}$,  it is
reasonable to require the reconstruction error to be smaller than
the signal level:
\begin{align*}
  \Exp{\Abs{\Paren{\West - \W}_{mn}}^2}^{1/2} \leq \tau_0 \frac{\CW^{\abs m + \abs n}}{\abs{m}^{\abs m}
    \abs{n}^{\abs n}}
\end{align*}
with $\tau_0>0$ being a tolerance number. Using
\eqref{eq:recon_error_bound_1} at the order $m=n=K$, this
implies 
\begin{align*}
  \snoise \frac{1}{\abs{d_K}} \lesssim \tau_0 \Paren{\frac{\CW}{K}}^{2K},
\end{align*}
which yields $\snoise K^{K+1/2} \lesssim \tau_0 \rho^{-K}$ by \eqref{eq:dn_asymptotic}.
Finally, according to the regime \eqref{eq:meas_noise} and by definition of the SNR, we find
\begin{align}
  \label{eq:max_resolving_order_bound}
  K^{K+1/2} \lesssim \tau_0 \SNR,
\end{align}
and the maximal resolving order is defined as the largest $K$ satisfying
\eqref{eq:max_resolving_order_bound}.

\section{Shape descriptor and identification in a dictionary}
\label{sec:Shape_Descriptor}

In this section we construct the shape descriptors based on the
scattering coefficients which are invariant to rigid transform.
For a collection of standard shapes, we build a
frequency-dependent dictionary of shape descriptors and develop a
shape identification algorithm in order to identify a shape from
the dictionary up to some translation, rotation and scaling. In
the following, we denote by $B$ a reference shape of size $1$ 
centered at the origin, so that the unknown inclusion $D$ is
obtained from $B$ by a rotation with angle $\theta$, a scaling
$s>0$ and a translation $z\in \R^2$ as $D = z+ s R_\theta B$.

\subsection{Translation and rotation invariant shape descriptor}
\label{sec:shape-descriptor}

The construction of the shape descriptor starts from the formula \eqref{eq:SCT-trl}, where the
convolution relationship suggests us to pass to the Fourier domain. It turns out that the Fourier
series of $W_{mn}[D]$ is just the far-field pattern $\Af_D$ \cite{ammari_enhancement_2011-2}. The following proposition holds.
\begin{prop}
  \label{prop:farfieldpat_SCT}
  Let $\xi= {(\xi_1,\xi_2)}^\top \in [0, 2\pi]^2$. Then, we have
  \begin{align}
    \label{eq:Fourier_SCT}
    \Af_D(\xi; \omega) = \sum_{m,n\in\Z} W_{mn}[D, \ve, \mu, \omega] e^{i m(\frac\pi 2
      -\xi_1)} e^{-i n(\frac\pi 2 - \xi_2)}.
  \end{align}
\end{prop}

\begin{rmk}
  For the circular acquisition system mentioned in Section \ref{sec:full-angle-view}, the modulus
  of the far-field pattern can be read off directly from the measurement when the radius $R$ and the
  number of sources $N_s$ and receivers $N_r$ are large, thanks to
  \eqref{eq:Vsr_far_field}. See Figure~\ref{fig:ffpattern_Vsr}.
\end{rmk}

\graphicspath{{./figures/}}
\def\figwidth{7cm}
\begin{figure}[htp]
  \centering
  \subfigure[$\Abs{V_{sr}}$]{\includegraphics[width=\figwidth]{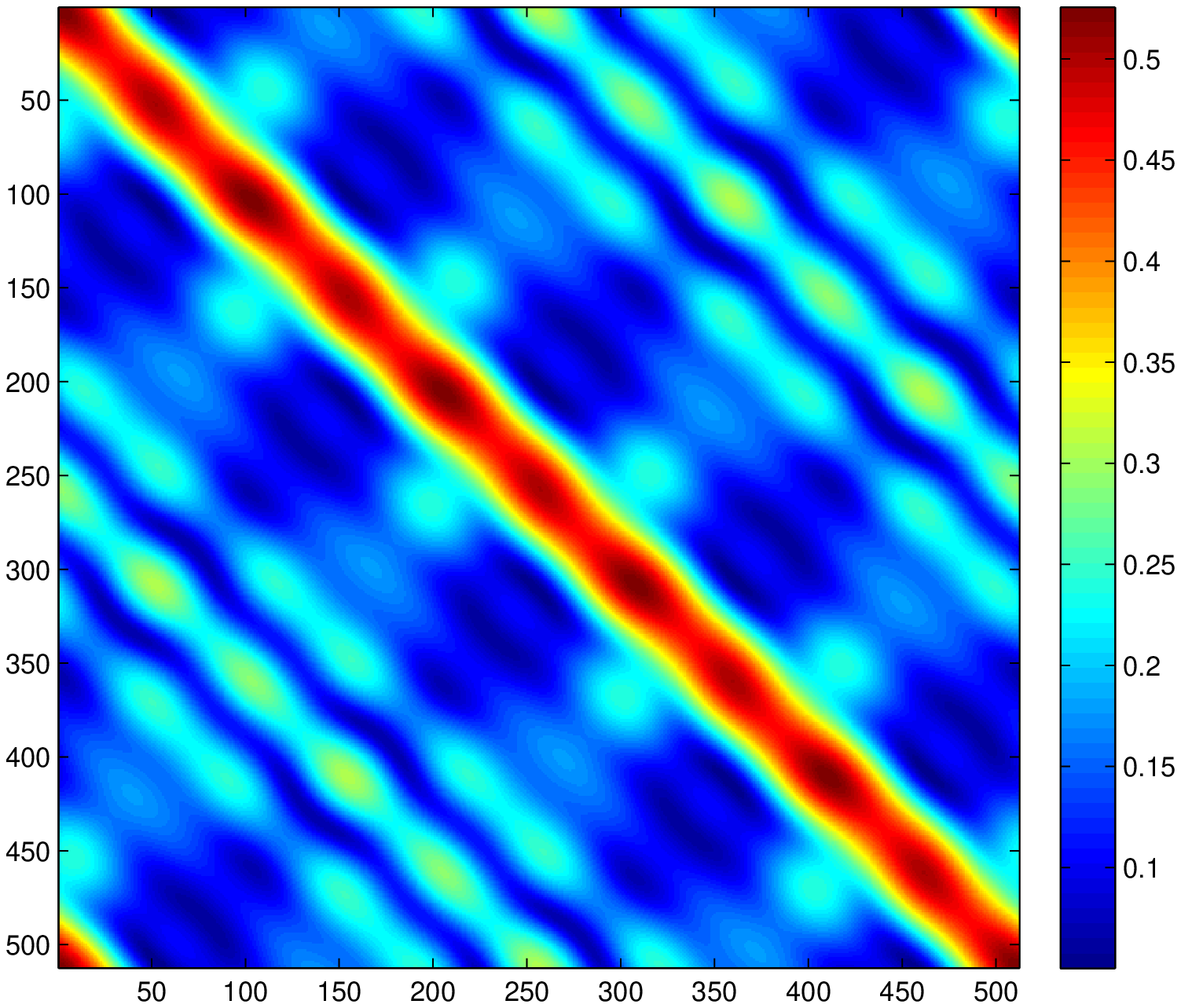}}
  \subfigure[$\Abs{\Af_D(\cdot;2\pi)}/\sqrt{8\pi k_0
    R}$]{\includegraphics[width=\figwidth]{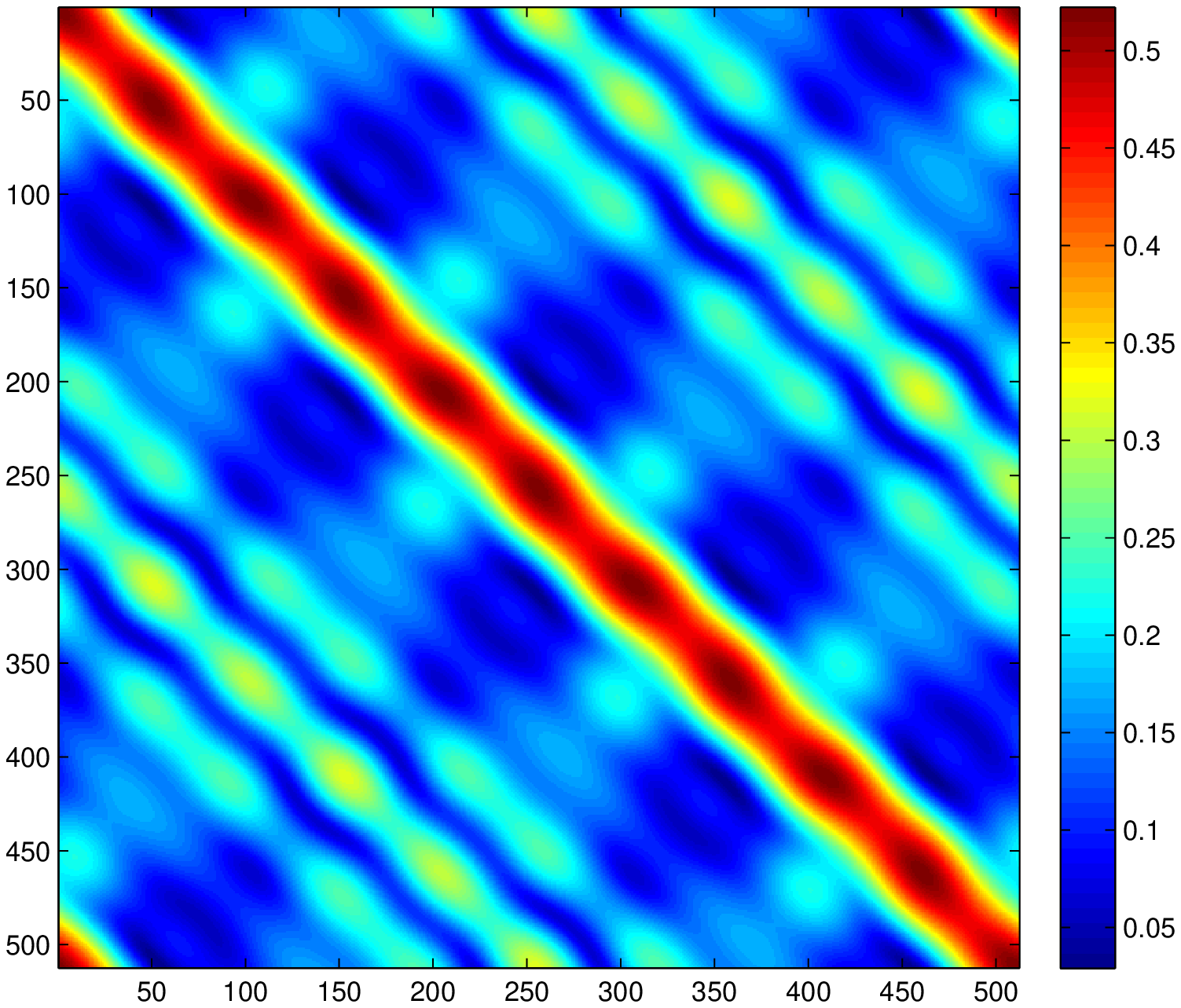}}
  \caption{(a) Modulus of the measurement of a flower-shaped object of size 1 with $N_s=512$ sources
    and $N_r=512$ receivers at the frequency $\omega=2\pi$ and the radius $R=10$, (b)
    $\Abs{\Af_D(\cdot;2\pi)}/\sqrt{8\pi k_0 R}$ computed via \eqref{eq:Fourier_SCT}. The relative
    difference between (a) and (b) is about $10^{-2}$.}
  \label{fig:ffpattern_Vsr}
\end{figure}

We emphasize that the Fourier series in \eqref{eq:Fourier_SCT} converges thanks to the decay
\eqref{eq:control_SCT} of $W_{mn}$.  Then using the translation formula \eqref{eq:SCT-trl}, we
establish the relations between the far-field pattern of $D$ and $B$.
\begin{prop}
  \label{prop:Fourier_SCT_DB}
  Let $D=z+s R_\theta B$. We denote by $\theta_z$ the angle of $z$ in polar coordinate, and define
  $\phi_z(\xi):=e^{ik_0\abs z \cos(\xi_1-\theta_z)} e^{-ik_0\abs z \cos(\xi_2-\theta_z)}$. Then, we have
  \begin{align}
    \label{eq:Fourier_SCT_DB}
    \Af_D(\xi; \omega) = \phi_z(\xi)\, \Af_B(\xi_\theta; s\omega) \ \text{ with } \xi_\theta:=\xi
    - {(\theta,\theta)}^\top.
  \end{align}
\end{prop}
\begin{proof}
  By the transformation formulas \eqref{eq:SCT-tsr}, we have
  \begin{align}
    \label{eq:Wmn_DB}
    W_{mn}[D, \ve, \mu, \omega] &= \sum_{a,b\in\Z} \cwv_a(z) \widebar{\cwv_b(z)}
    W_{m-a,n-b}[s R_\theta B,\ve, \mu, \omega] \notag \\
    &= \sum_{a,b\in\Z} \cwv_a(z) \widebar{\cwv_b(z)} e^{i(m-a)\theta}e^{-i(n-b)\theta}
    W_{m-a,n-b}[B, \ve, \mu, s\omega].
  \end{align}
  Using the Jacobi-Anger expansion, the Fourier series of the cylindrical wave reads:
  \begin{align*}
    \sum_{a\in\Z} \cwv_a(z) e^{i a (\frac \pi 2 - \xi_1)} = \sum_{a\in\Z} J_a(k_0\abs z) e^{i a
      (\frac \pi 2 - (\xi_1-\theta_z))} = e^{i k_0\abs z
      \cos(\xi_1-\theta_z)}.
  \end{align*}
  Therefore, computing the Fourier series on the both sides of \eqref{eq:Wmn_DB} yields
  \eqref{eq:Fourier_SCT_DB}.
\end{proof}


Hereafter, we illustrate the descriptor construction based on the scattering amplitude. 
We define the following quantity:
\begin{align}
  \label{eq:Invar_S}
  \mathcal S_D(v; \omega) := \int_{[0,2\pi]^2} \Abs{\Af_D(\xi; \omega){\Af_D(\xi-v; \omega)}} d\xi
  .
\end{align}
More precisely, taking the modulus on $\Af_D$ removes the effect of the translation, and the
auto-correlation in $\xi$ removes that of the rotation. In fact, using \eqref{eq:Fourier_SCT_DB}:
\begin{align}
  \label{eq:Shapedescrp_DB}
  \mathcal S_D(v; \omega) = \int_{[0,2\pi]^2} \abs{\Af_B(\xi_\theta; s\omega)
    \Af_B(\xi_\theta-v; s\omega)} d\xi = \mathcal S_B(v; s\omega),
\end{align}
where the last identity comes from the periodicity in $\xi$ of $\Af_B(\xi; \omega)$. We name
\eqref{eq:Invar_S} the \emph{shape descriptor}. The shape descriptor $\mcl S_D(v;\omega)$ is invariant to any translation and rotation.
Remark that for fixed $\omega$, it carries the information of the far-field
pattern of the shape, while for fixed $v\in[0,2\pi]^2$,  it shows how the far-field pattern varies as a function of the frequency.

\subsection{Scale estimation}
\label{sec:estim-scal}

From $\mcl S_D(v;\omega)$ we could hopefully construct, at least formally, some quantities $\mcl
I_D(v)$ being further invariant to the scaling, and obtain in that way shape descriptors easily
adapted for the purpose of identification like in \cite{ammari_target_2012}. For example, one may
consider the Hilbert transform in $\omega$ of $\mcl S_D$:
\begin{align}
  \label{eq:Hilbert_t}
  \mathcal{I}_D(v) := \int_0^{+\infty} {\frac{\mcl S_D(v; \omega)}{\omega}}
  d\omega,
\end{align}
which is invariant to any translation, rotation and scaling, since
\begin{align}
  \label{eq:Invar_tsr_Hilbert}
  \mathcal{I}_D(v) = \int_0^{+\infty} {\frac{\mcl S_B(v; s\omega)}{\omega}} d\omega = \int_0^{+\infty}
  {\frac{\mcl S_B(v; \omega)}{\omega}} d\omega = \mcl I_B(v).
\end{align}

Not to mention the well-definedness of \eqref{eq:Hilbert_t}, let
us point out immediately the physical infeasibility of such kind
of quantities, whatever the method in use. In fact, the frequency
$\omega$ is coupled with the scaling factor $s$ in the relation
\eqref{eq:Shapedescrp_DB}, hence if the knowledge of $\mcl S_D(v;
\omega)$ in the frequency $\omega$ is bounded away from $0$ or
infinity, then it is not possible to construct an $\mcl I_D$ being
invariant to arbitrary scaling factor $s\in\R^+$, which would
require the frequency at 0 or infinity.


\subsubsection{Determine $s$ by lookup table}
\label{sec:determ-lookup-table}

We assume that the physical operating frequency $0<\omin\leq \omega \leq \omax<\infty$, and $0<\smin
\leq s \leq \smax<\infty$ which means the inclusions that we are interested in should not be too
small nor to large. Then we can estimate $s$ by solving
\begin{align}
  \label{eq:optim_scaling_matching}
  \sest = \argmin_{s\in[\smin, \smax]} \Set{\int_\omin^\omax\Paren{\int_{[0,2\pi]^2} [\mcl S_D(v;\omega)-\mcl
    S_B(v;s\omega)]\, dv}^2 d\omega},
\end{align}
which is a lookup table technique since we are seeking the best value of $s$ in the table
$\int_{\sss{[0,2\pi]^2}} \mcl S_B(v; s\omega) dv$ indexed by $s$, by comparing with
$\int_{\sss{[0,2\pi]^2}} \mcl S_D(v; \omega) dv$. This idea is illustrated in Figure~\ref{fig:SDSB}.
Remark that the functional in \eqref{eq:optim_scaling_matching} is minimized at the value of the
true scaling factor, although the minimizer might not be unique.

\graphicspath{{./figures/}}
\def\figwidth{7cm}
\begin{figure}[htp]
  \centering
  \subfigure[]{\includegraphics[width=\figwidth]{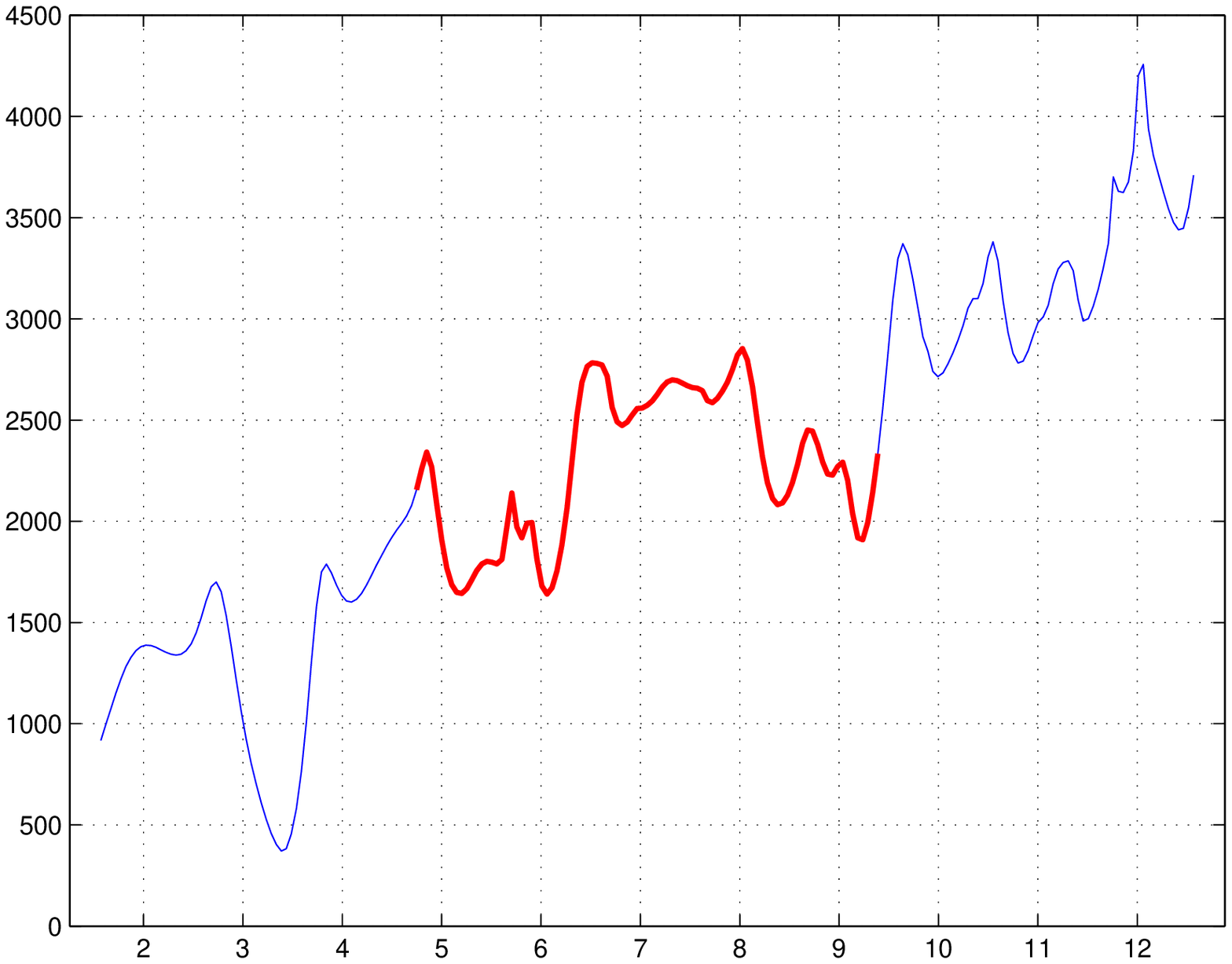}}
  \subfigure[]{\includegraphics[width=\figwidth]{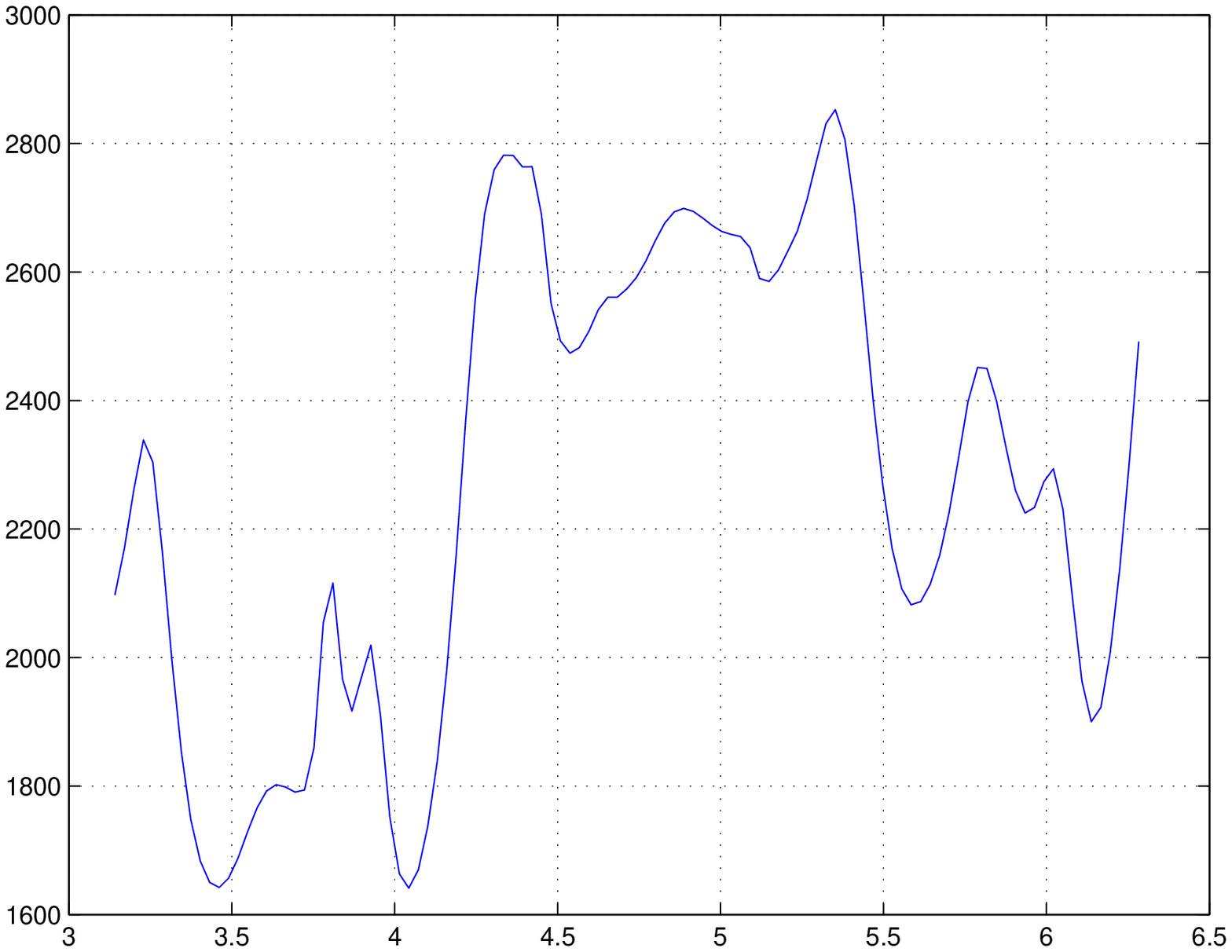}}
  \caption{Shape descriptors of a flower-shaped object $B$ and of $D=1.5\times B$, as functions of
    the frequency $\omega$. (a) $\int_{\sss{[0,2\pi]^2}} \mcl S_B(v; \omega) dv$ with
    $\omega\in[0.5\pi, 4\pi]$. (b) $\int_{\sss{[0,2\pi]^2}} \mcl S_D(v; \omega) dv$ with
    $\omega\in[\pi, 2\pi]$. The red part in (a) corresponds to the frequency $\omega\in[1.5\pi,
    3\pi]$.}
  \label{fig:SDSB}
\end{figure}

\subsubsection{Numerical implementation of the scaling estimation}
\label{sec:numer-impl}

Numerically, \eqref{eq:optim_scaling_matching} can be solved approximately by sampling. Let $\Nwdic,
\Nw, \Nv, N_s$ be positive integers. We define $\set{\omegar_l, l=0\ldots \Nwdic}$ uniformly
distributed points on $[\ormin, \ormax]$, with $\ormin:=\omin\smin, \ormax:=\omax\smax$. Similarly,
let $\set{\omega_k, k=0\ldots \Nw}$ be uniformly distributed on $[\omin, \omax]$, and
$\set{(v^1_i,v^2_j)^\top, i,j=1\ldots \Nv}$ be uniformly distributed in $[0,2\pi]^2$. Then we sample
the functions $\mcl S_B$ and $\mcl S_D$ at discrete positions as follows:
\begin{align}
  \label{eq:SD_ijk_SB_ijl}
  \mcl S^D_{ijk}:= \mcl S_D({(v^1_i, v^2_j)}^\top; \omega_k), \ \mcl S^B_{ijl}:= \mcl S_B({(v^1_i,
    v^2_j)}^\top; \omegar_l).
\end{align}
For $\set{s_t, t=0\ldots, N_\delta}$ uniformly distributed in $[\smin, \smax]$, we can approximate the
functional inside the argmin in \eqref{eq:optim_scaling_matching} by a discrete version:
\begin{align}
  \label{eq:optim_scaling_discrete}
  \mcl J(t; D, B) = \sum_{k=0}^{\Nw} \sum_{l\in
      I_k(s_t)} \Paren{\sum_{i,j=1}^\Nv \Paren{\mcl S^D_{ijk} - \mcl S^B_{ijl}}}^2
\end{align}
with the index set $I_k(s):=\set{1\leq l \leq \Nwdic, \text{ s.t. } \omegar_{l-1}\leq s\omega_k \leq
  \omegar_l}$. Finally, the scaling is estimated through
\begin{align}
  \label{eq:err_DB}
  \ve(D,B) = \min_{t=0\ldots N_\delta} \mcl J(t; D, B).
\end{align}

\subsection{Frequency-dependent dictionary and identification algorithm}
\label{sec:freq-depend-dict}

For a collection of standard shapes $\set{B_n}_{n=1\ldots N}$, we
pre-compute the discrete samples $\Paren{\mcl
S^{B_n}_{ijl}}_{ijl}$ of the shape descriptor $\mcl
S_{B_n}(v;\omega)$ as in \eqref{eq:SD_ijk_SB_ijl} for
$v\in[0,2\pi]^2$ and $\omega\in[\ormin, \ormax]$. This constitutes
a frequency-dependent dictionary of shape descriptors.


Suppose $D$ is the realization of a shape from the dictionary
$\set{B_n}_n$, up to some unknown translation, rotation and
scaling, and the scaling factor satisfies $\smin\leq s\leq \smax$.
In order to identify $D$ from the dictionary, we compute the
discrete samples $\Paren{\mcl
  S^{D}_{ijk}}_{ijk}$ of the shape descriptor $\mcl S_D(v;\omega)$ as in \eqref{eq:SD_ijk_SB_ijl},
and calculate $\ve(D, B_n)$ in \eqref{eq:err_DB} for all shapes of the dictionary. The true shape is expected to give the best estimation of scaling and to minimize the error $\ve(D,B_n)$ among all the
dictionary elements, thus we take the minimizer of $\set{\ve(D,B_n)}_n$ as the identified shape. This
procedure is described in Algorithm \ref{alg:shape_identify}.
\begin{algorithm}[H]
  \caption{Shape identification algorithm}
  \label{alg:shape_identify}
  \begin{algorithmic}
    \STATE Input: {$\Paren{\mcl S^{D}_{ijk}}_{ijk}$ of unknown shape $D$, $\dicoSB$ of the whole dictionary}
    \FOR{$B_n$ in the dictionary}
    \STATE $\ve_n \leftarrow \ve(D, B_n)$;
    \STATE $n \leftarrow n+1$;
    \ENDFOR
    \STATE Output: {The true dictionary element $n^* \leftarrow \argmin_n \ve_n$.}
  \end{algorithmic}
\end{algorithm}

\subsection{Shape descriptor with partial far-field pattern}
\label{sec:shape-descr-partial}

Let us mention an interesting property of $\mcl S_D$ at the end of this section. Given $0<\alpha\leq
2\pi$, we denote the partial far-field pattern $\Af_D$
\begin{align}
  \label{eq:partial_far_field}
  \wt{\Af_D}(\xi;\omega):={\Af_D}(\xi;\omega)\mathbbm{1}(\abs{\xi_1-\xi_2}\leq \alpha),
\end{align}
\ie, only the band diagonal of width $\alpha$ in $\Af_D(\cdot; \omega)$ is available. Since
\begin{align*}
  \mathbbm{1}(\abs{\xi_1-\xi_2}\leq \alpha) = \mathbbm{1}(\abs{(\xi_1-\theta)-(\xi_2-\theta)}\leq \alpha),
\end{align*}
so for the same reason as in \eqref{eq:Shapedescrp_DB} we deduce that
\begin{align*}
  \int_{[0,2\pi]^2} \Abs{\wt{\Af_D}(\xi; \omega) \wt{\Af_D}(\xi-v; \omega)} d\xi =
\int_{[0,2\pi]^2} \Abs{\wt{\Af_B}(\xi_\theta; s\omega) \wt{\Af_B}(\xi_\theta-v; s\omega)} d\xi.
\end{align*}
Therefore, the shape descriptor $\mcl S_D$ computed with $\wt{\Af_D}$ is still invariant to
translation and rotation. Assuming that the dictionary $\dicoSBn$ is also computed using partial far
field pattern $\wt{\Af_{B_n}}$ of the same constant $\alpha$, then one can always apply
Algorithm~\ref{alg:shape_identify} for shape identification.

\section{Numerical experiments}
\label{sec:numer-exper}

In the rest of this paper, we present numerical results
demonstrating the theoretical framework presented in previous
sections. Given a shape $D$, the overall procedure of a numerical
experiment consists of the following steps:
\begin{enumerate}
\item Data simulation. The MSR matrix $\V$ is simulated for a range of frequency $[\omin, \omax]$ by
  evaluating the integral representation \eqref{eq:solu}. Then the white noise is of level
  \begin{align*}
    \snoise = \sigma_0 \norm{\V}_F/\sqrt{N_sN_r}
  \end{align*}
  is added as in \eqref{eq:AWBE_mat_noise}. $\sigma_0$ is the percentage of the noise in the
  measurement and it is proportional to $\SNR^{-1}$.
\item Reconstruction of scattering coefficients. For each frequency, we reconstruct $\W$ either by
  solving the least-squares problem \eqref{eq:least_square} (or directly by the formula
  \eqref{eq:L_pseudoinv} in case of the full aperture of view). This step is skipped in the case of
  limited aperture of view, see Section~\ref{sec:shape-ident-with-lim-view}.
\item Shape identification. We calculate the shape descriptors and follow the procedure described in
  Algorithm~\ref{alg:shape_identify} for identification.
\end{enumerate}

Our dictionary of shapes consists of eight elements as shown in
Figure~\ref{fig:8shapes}. All shapes share the same permittivity
$\ve_*=3$ and the permeability $\mu_*=3$, while the values of the
background are $\ve_0=1, \mu_0=1$. The frequency-dependent
dictionary of shape descriptors $\dicoSB$ is computed for the
frequency range $[\ormin,\ormax]=[0.5\pi,4\pi]$ with $\Nwdic=219$
and $\Nv=512$. The data are simulated for the operating frequency
range $[\omin,\omax]=[\pi, 2\pi]$ with $\Nw=109$ and the valid
scaling range is $[\smin, \smax]=[0.5, 2]$ with $N_\delta=751$.

\graphicspath{{./figures/dico/}}
\def\figwidth{3cm}

\begin{figure}[htp]
  \centering
  \subfigure[Ellipse]{\includegraphics[width=\figwidth]{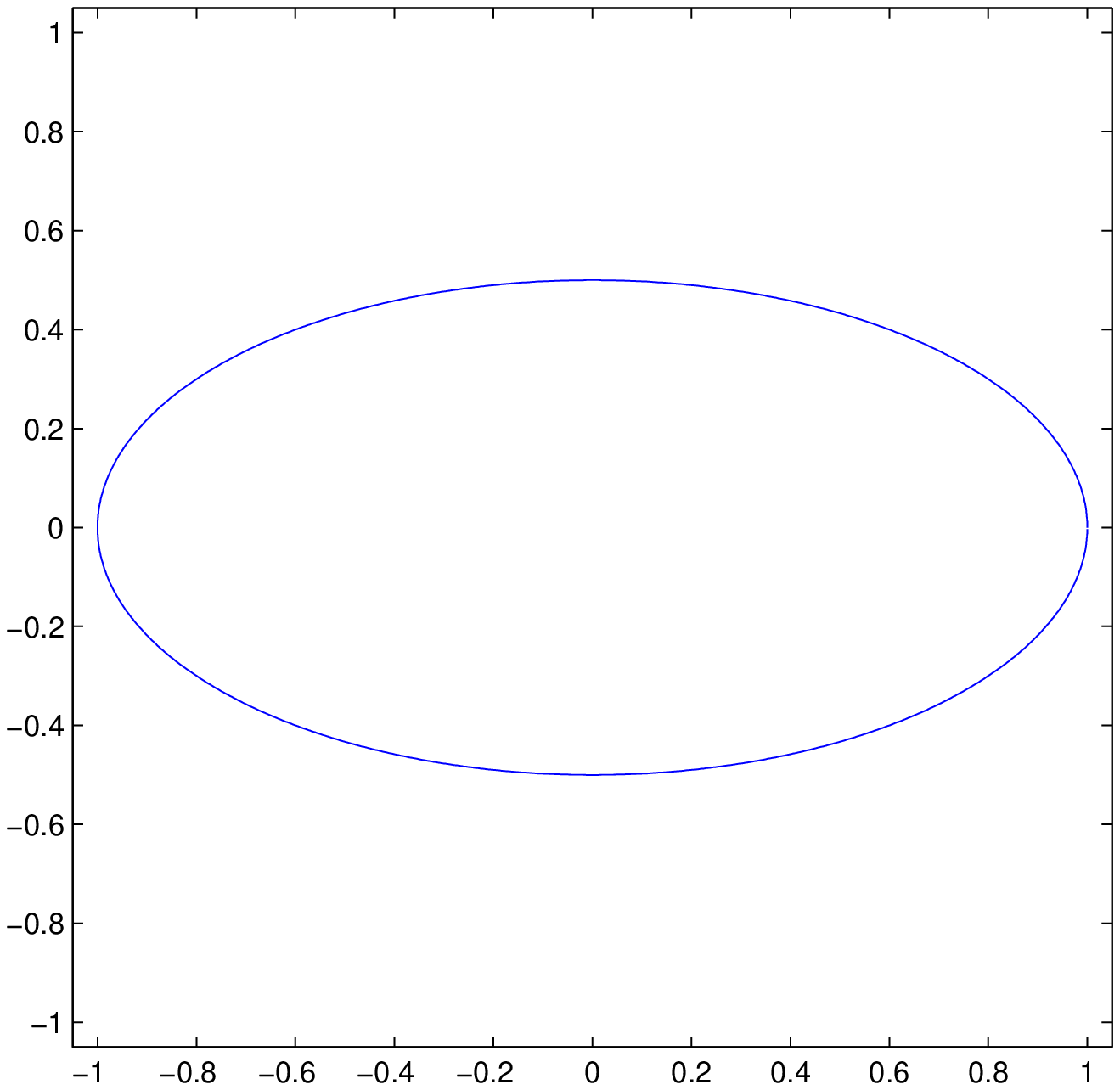}}
  \subfigure[Flower]{\includegraphics[width=\figwidth]{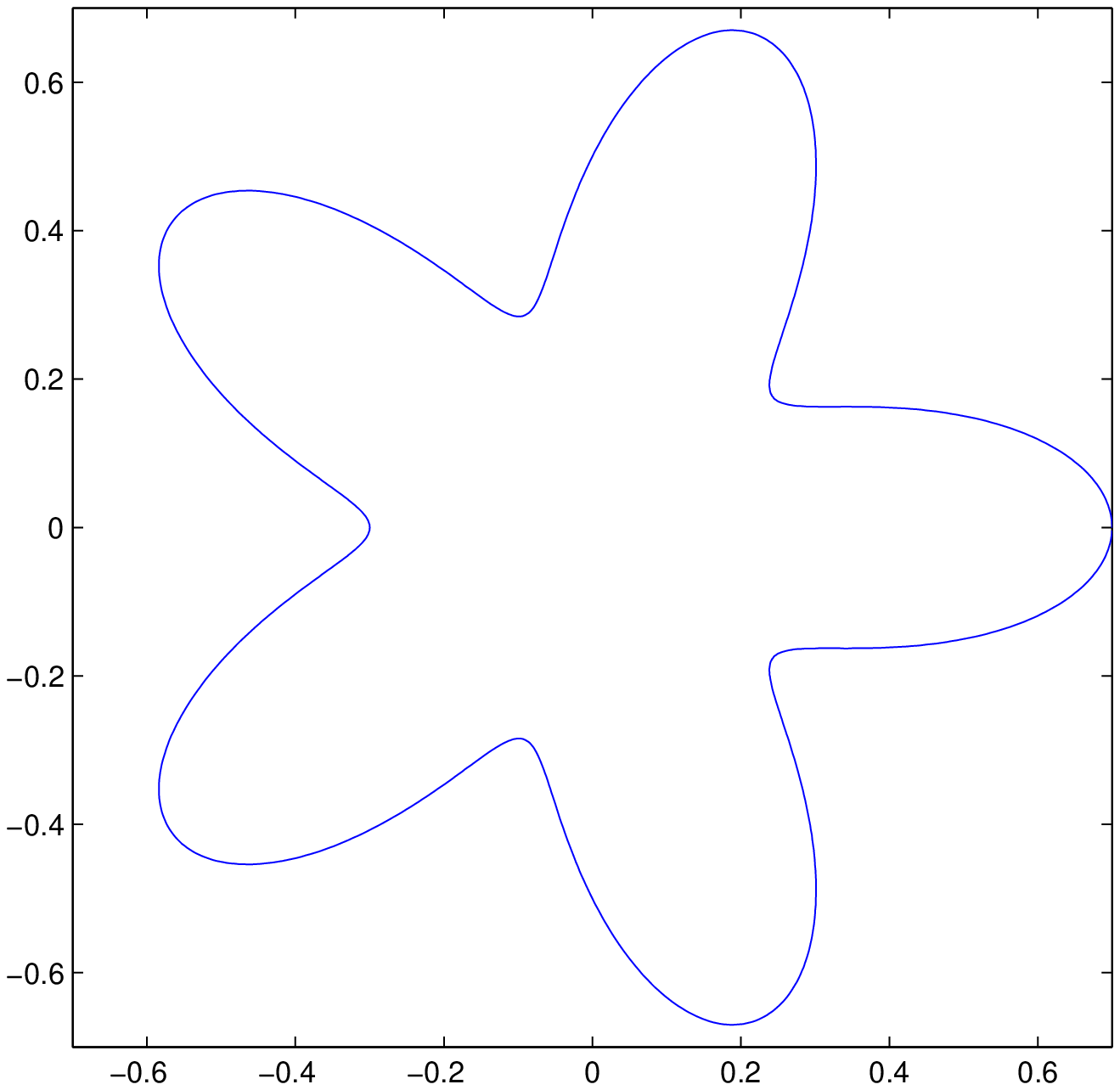}}
  \subfigure[Letter A]{\includegraphics[width=\figwidth]{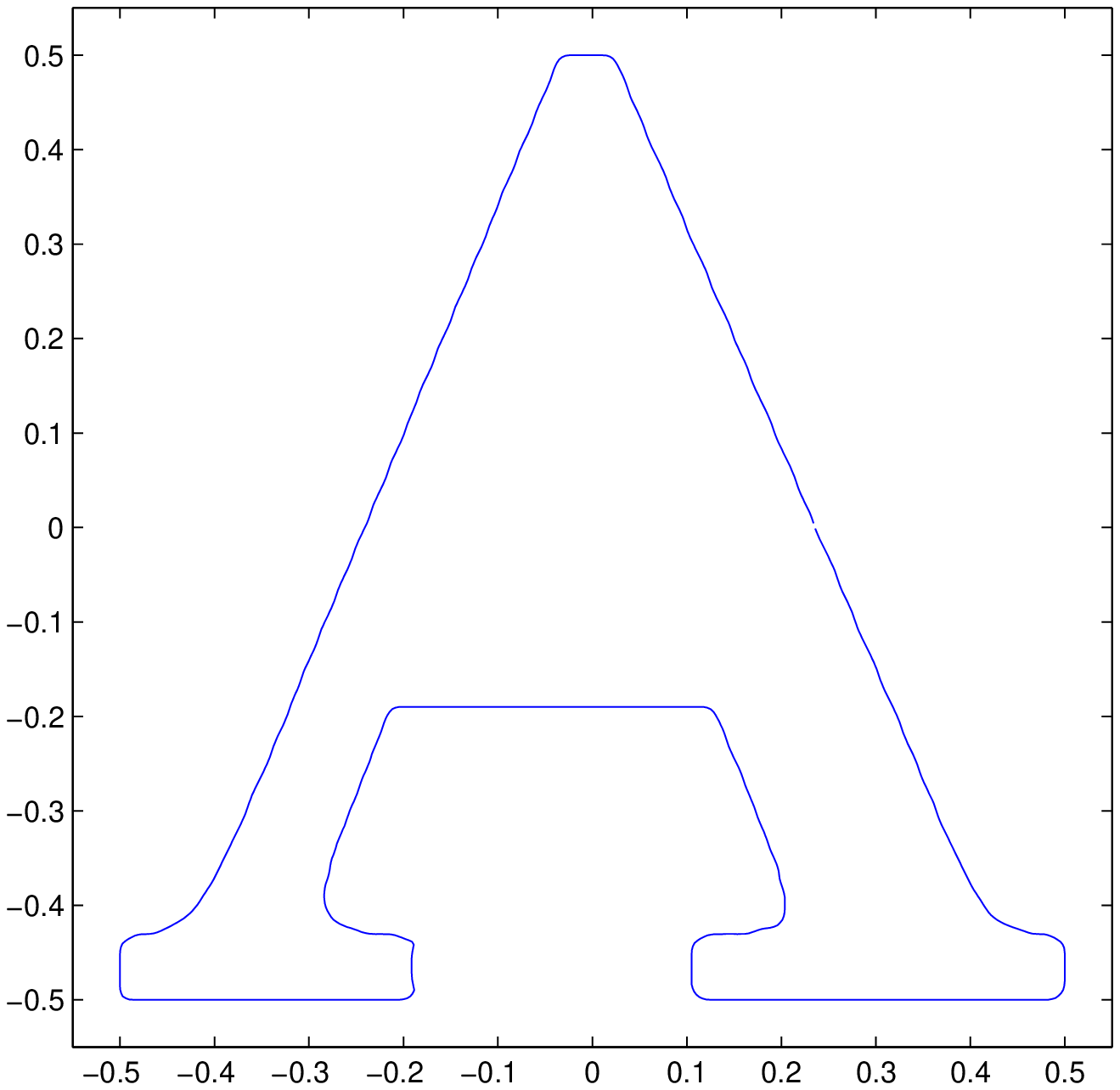}}
  \subfigure[Square]{\includegraphics[width=\figwidth]{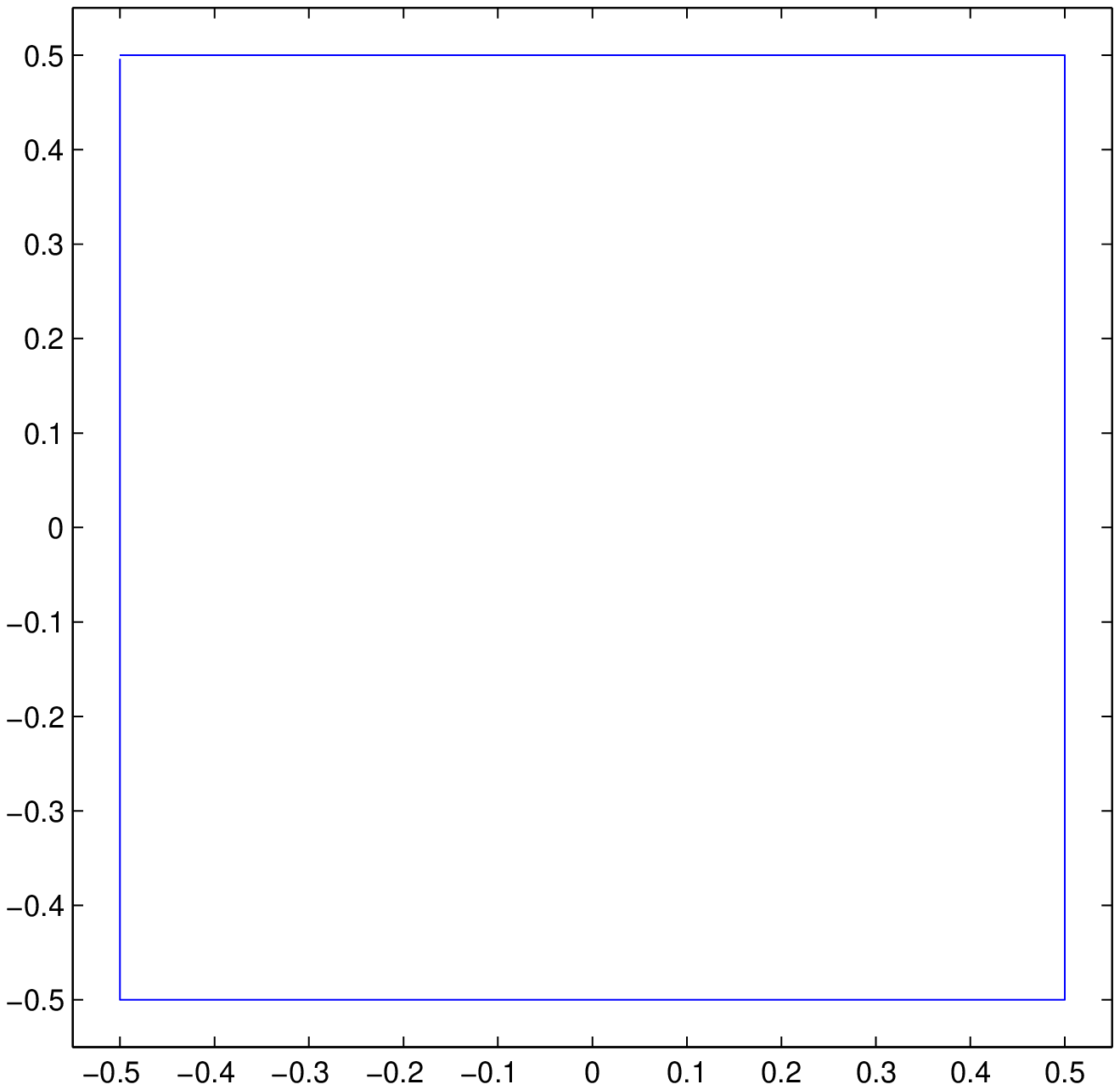}}
  \subfigure[Letter E]{\includegraphics[width=\figwidth]{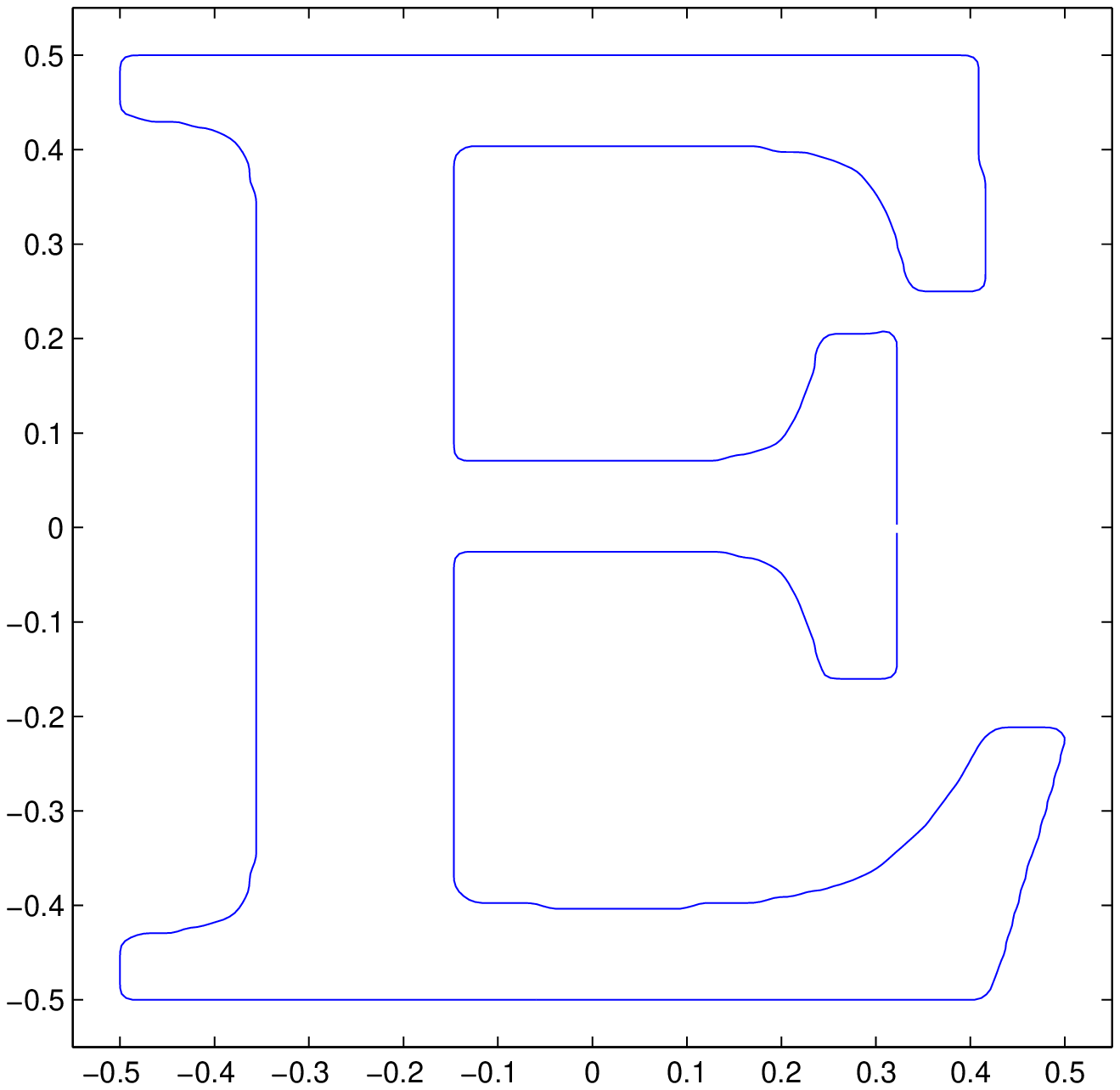}}
  \subfigure[Rectangle]{\includegraphics[width=\figwidth]{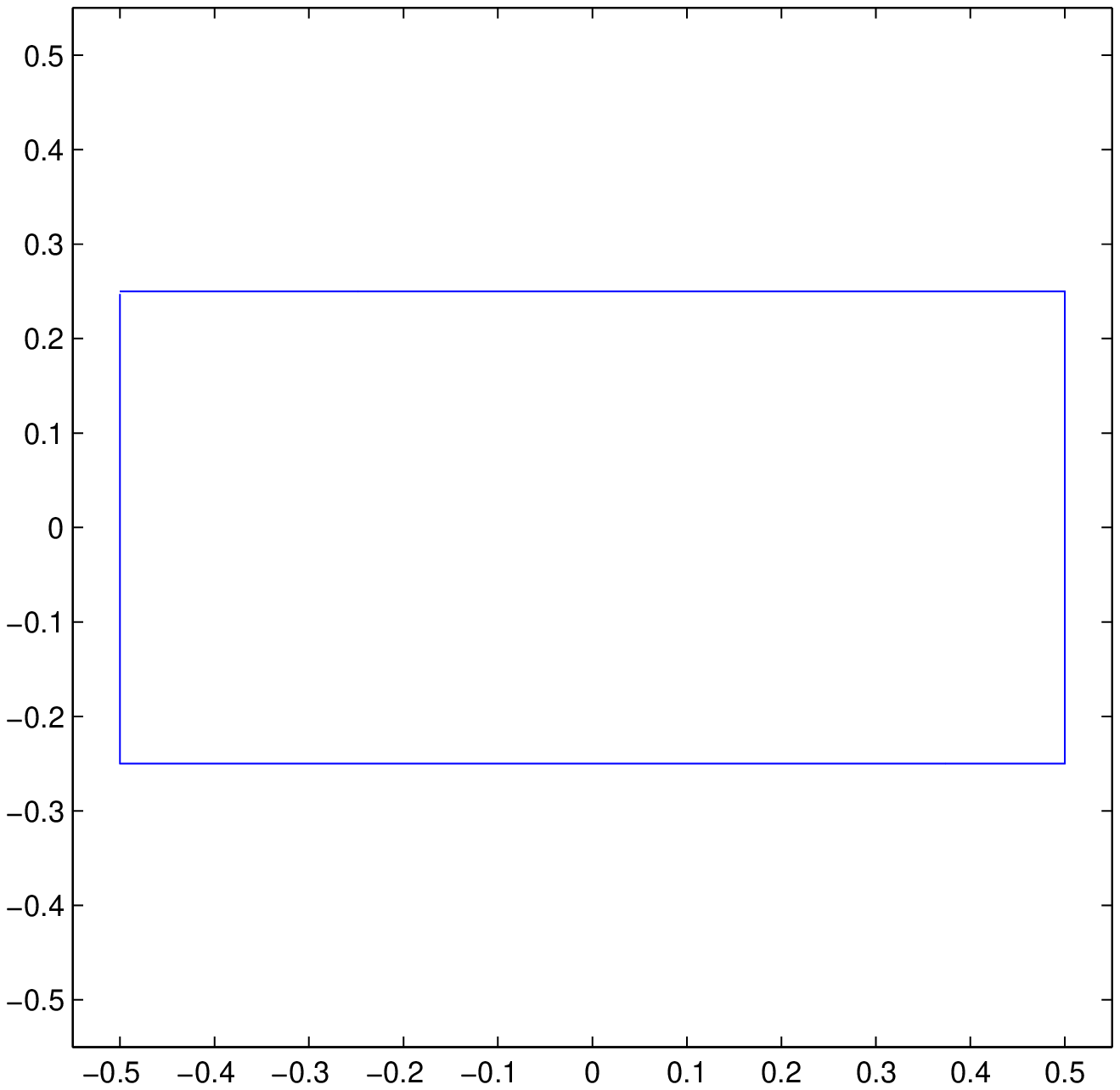}}
  \subfigure[Disk]{\includegraphics[width=\figwidth]{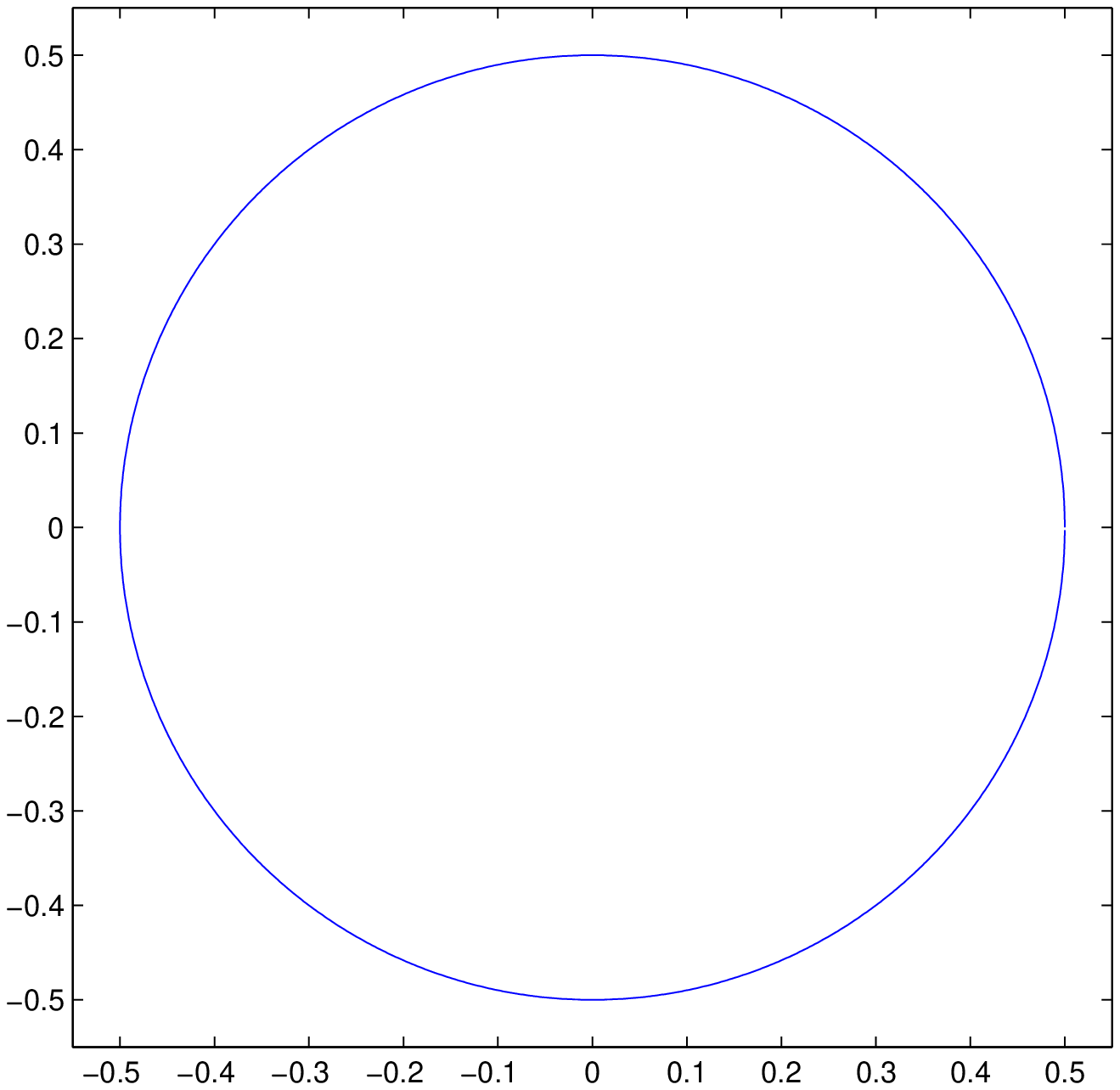}}
  \subfigure[Triangle]{\includegraphics[width=\figwidth]{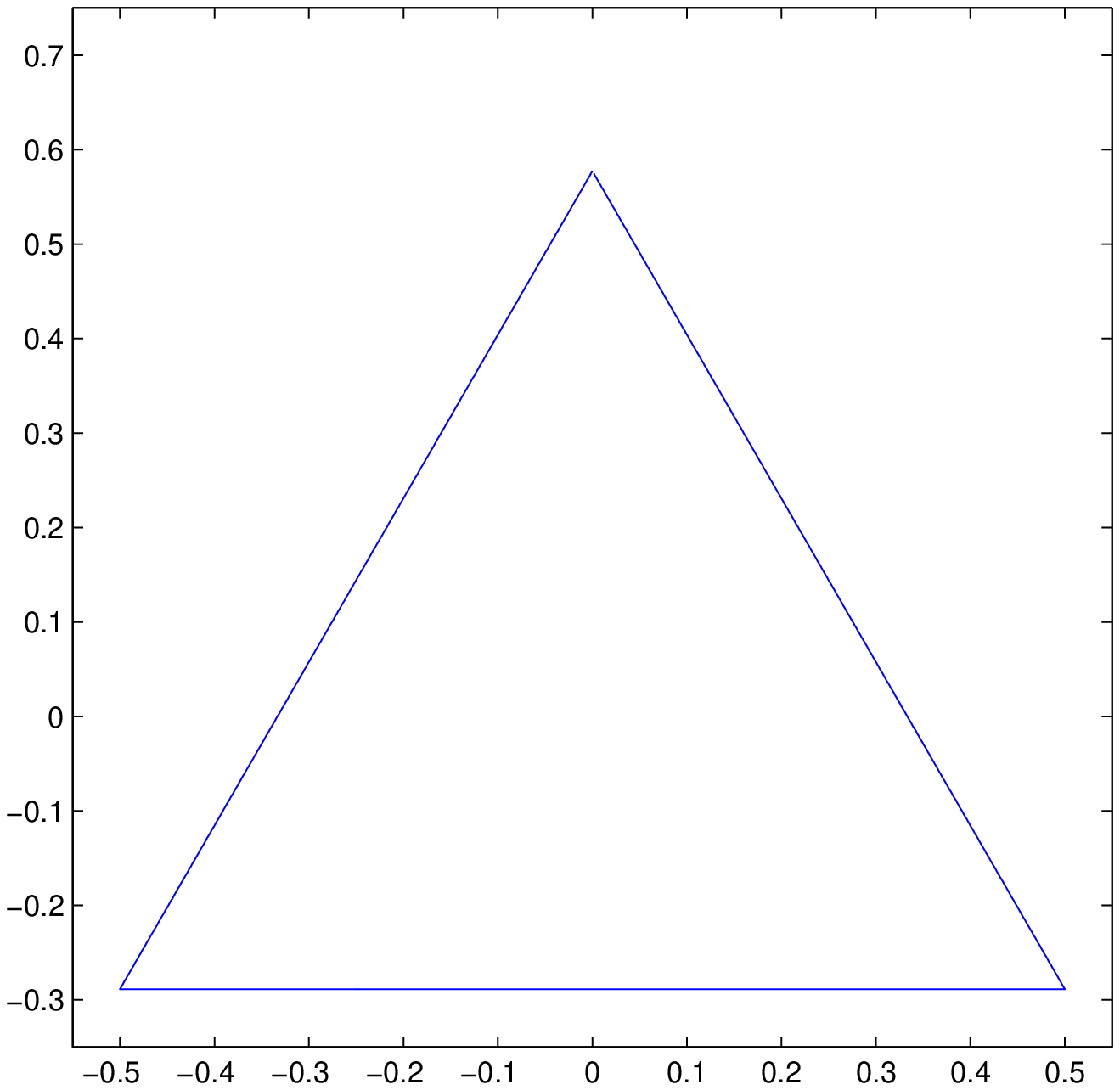}}
  \caption{A small dictionary of shapes.}
  \label{fig:8shapes}
\end{figure}

\subsection{Experiment and parameter settings}
\label{sec:exp-parm-settings}

We generate a circular acquisition system like in Section~\ref{sec:full-angle-view} with the radius
$R=3$, and $N_s=91$ plane waves as sources and $N_r=91$ receivers. For reason of simplicity, we
always choose the center $z_0=[0,0]^\top$. Figure~\ref{fig:acqsys} (a) illustrates this acquisition
system.

Another possibility is shown in Figure~\ref{fig:acqsys} (b), where the sources and receivers are
divided into different groups of aperture angle $\alpha$ such that no intercommunication exists between
groups. Such an acquisition is close to that used by the bat. Each group in Figure~\ref{fig:acqsys}
(b) represents the spatial position of a flying bat's body, which sends plane waves with limited
aperture of wave direction and receives the scattered field via receptors on its body. By flying
around the target and taking measurement at many positions, the bat actually acquires data
corresponding to the band diagonal part of the MSR matrix of a full aperture of view system, and the
width of the band diagonal is the number of the receivers (or proportional to $\alpha$).

\graphicspath{{./figures/}}
\def\figwidth{6.5cm}
\begin{figure}[htp]
  \centering
  \subfigure[Full view]{\includegraphics[width=\figwidth]{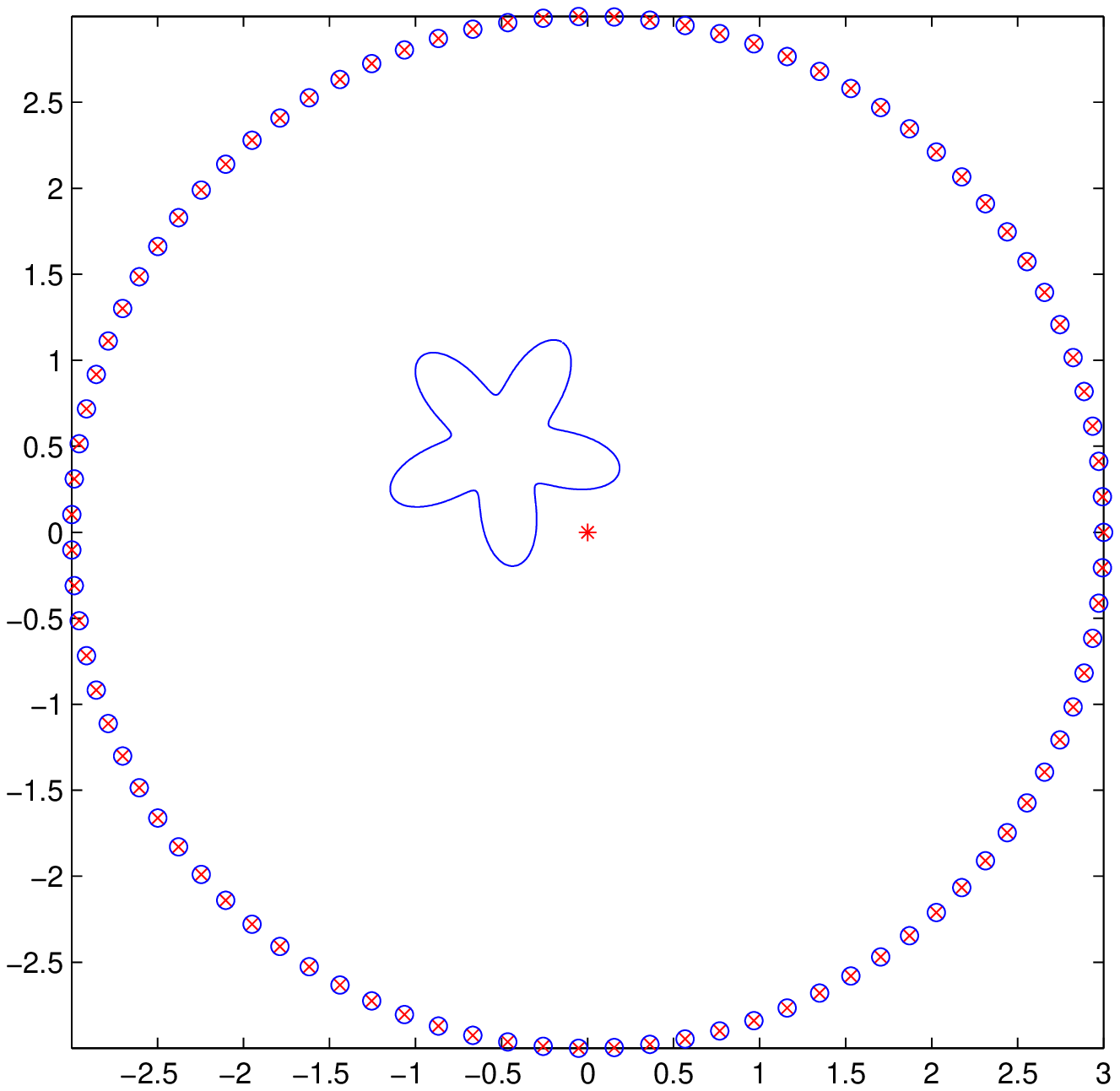}}
  \subfigure[Limited view]{\includegraphics[width=\figwidth]{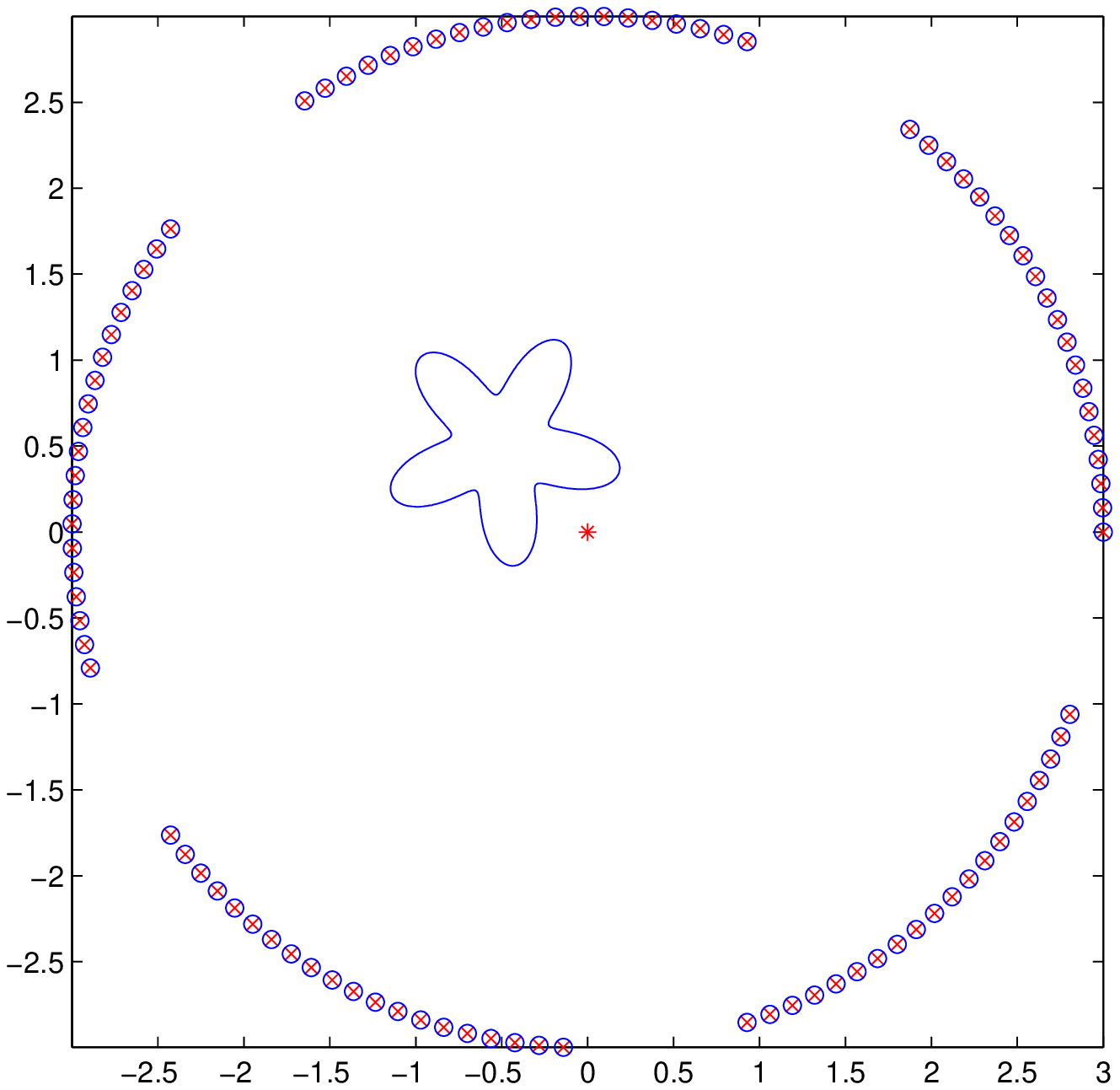}}
  \caption{Circular acquisition systems. (a) Full aperture with $N_s=91$ plane wave sources (with
    the angular position marked by 'o') and $N_r=91$ receivers (marked by 'x'). (b) Limited aperture
    with the sources and receivers divided into 5 groups, and the receivers of one group can only
    see the sources of the same group. The measurement center $z_0$ in both cases is marked by '*'.}
  \label{fig:acqsys}
\end{figure}

Figure~\ref{fig:svd_acqsys} shows the singular values of the
associated operators $\bL$ at the operating frequency
$\omega=2\pi$. The stairwise distribution and the increase of
singular values with the order $K$ in Figure~\ref{fig:svd_acqsys}
(a) confirm the theoretical result in \eqref{eq:sing_vl}. It can
be seen that at the low-order (\eg $K\leq 30$) the operator $\bL$
is numerically well-conditioned. On the contrary, as shown in
Figure~\ref{fig:svd_acqsys} (b) the situation in the limited
aperture of view is dramatically different: the operator is
ill-conditioned even for very low-order (\eg, $K=5$), which means
the reconstruction of scattering coefficients is highly unstable
with such an acquisition
system. 

\def\figwidth{7cm}
\begin{figure}[htp]
  \centering
  \subfigure[Full view]{\includegraphics[width=\figwidth]{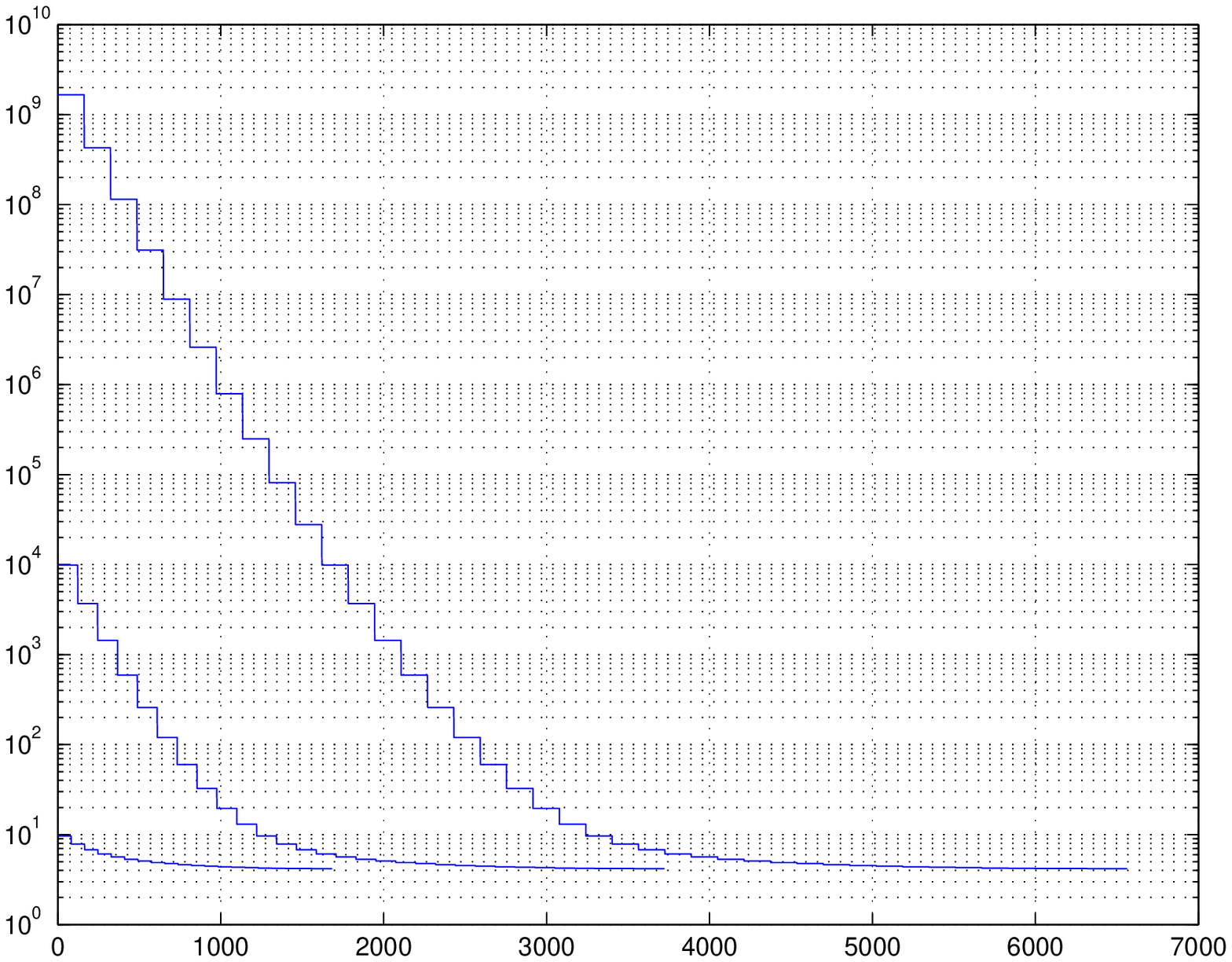}}
  \subfigure[Limited view]{\includegraphics[width=\figwidth]{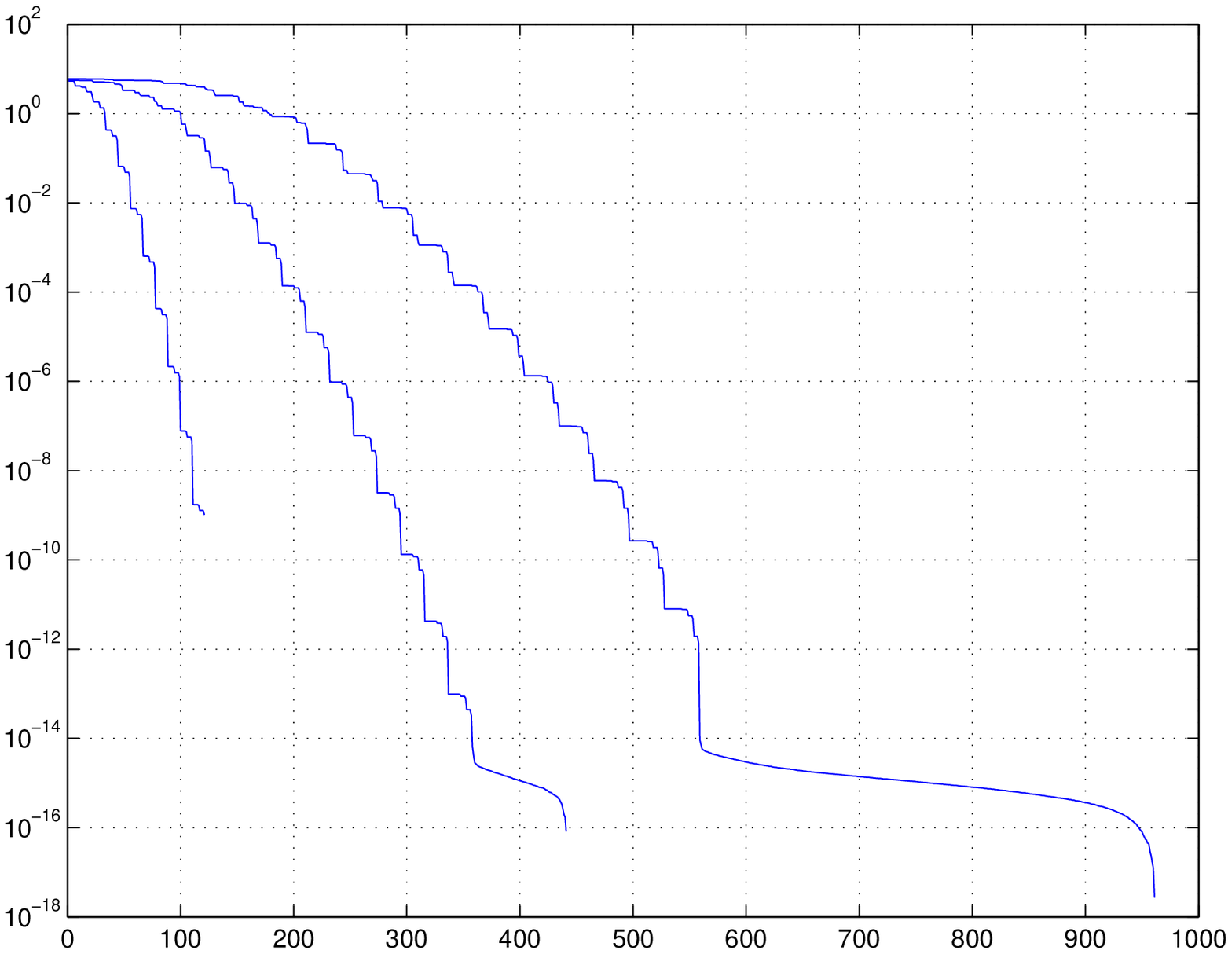}}
  \caption{Singular values of the operator $\bL$ of the acquisition systems of Figure~\ref{fig:acqsys}
    (a) and (b) at the frequency $\omega=2\pi$. The curves from top to bottom correspond to the
    operators $\bL$ of order $K=40, 30, 20$ in (a), and $K=15, 10, 5$ in (b), respectively. }
  \label{fig:svd_acqsys}
\end{figure}

\subsubsection{Resolving order $K$}
The maximal resolving order $K$ given by
\eqref{eq:max_resolving_order_bound} is an asymptotic bound which
holds when both the radius $R$ and $K$ are large, and it might be
too pessimistic for the numerical range that we are interested in.
In practice, the maximal resolving order can be determined by a
tuning procedure.  For the system in Figure~\ref{fig:acqsys} (a), we
vary the noise level and reconstruct the matrix $\West$ at
different truncation order $K$. Figure~\ref{fig:resolv_ord} plots
the relative error of the least-squares reconstruction
$\norm{\West-\W}_F/\norm{\W}_F$ as a function of $K$. It can be
seen that in the full aperture case, the reconstruction is rather
robust and with $20\%$ of noise the resolving order $K$ can go
beyond $30$ by bearing only $10\%$ of error. On the contrary, the
limited aperture of view can not provide any stable
reconstructions even at very low noise level and small truncation
order.

\def\figwidth{7cm}
\begin{figure}[htp]
  \centering
  \subfigure[Full view]{\includegraphics[width=\figwidth]{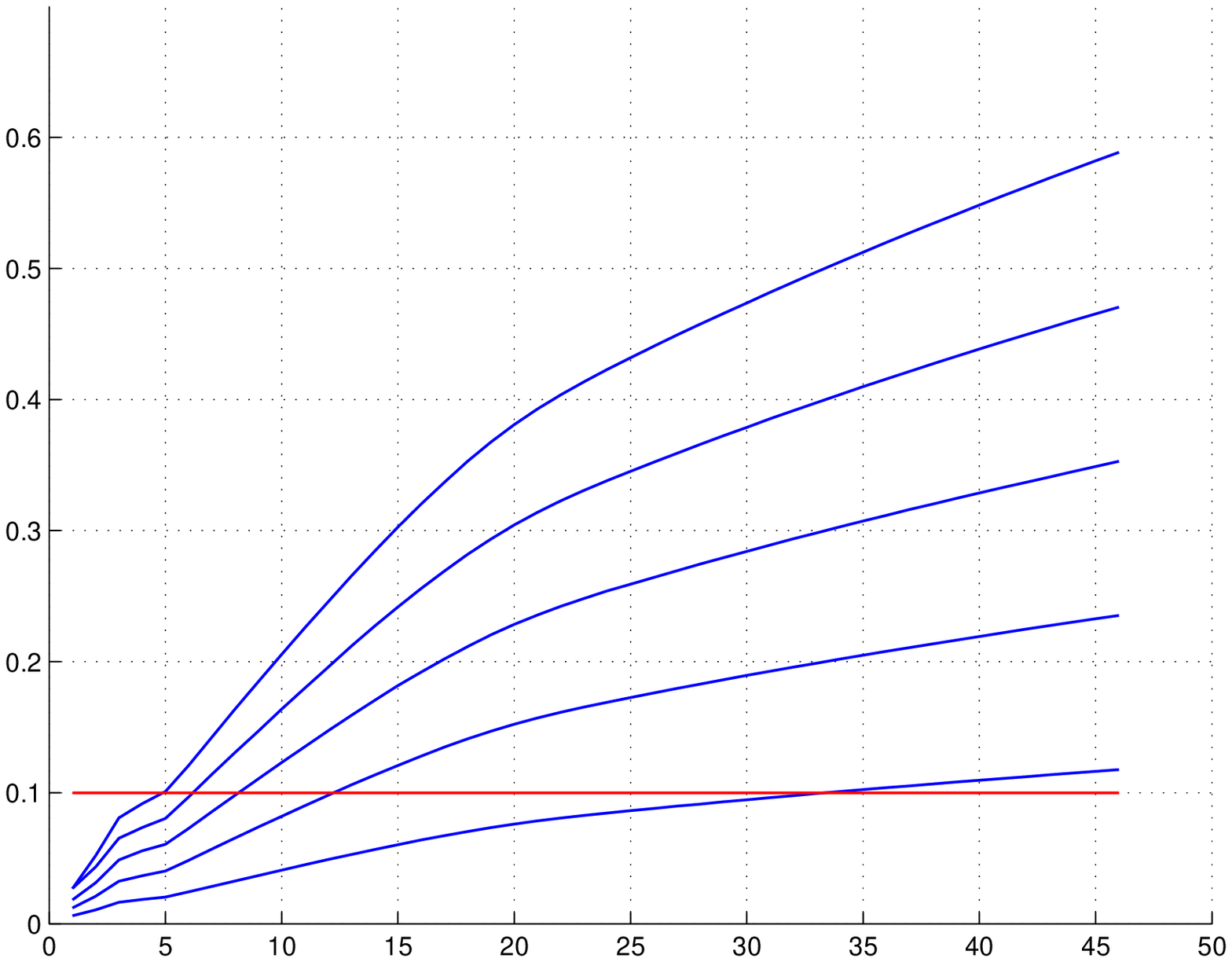}}
  \subfigure[Limited view]{\includegraphics[width=\figwidth]{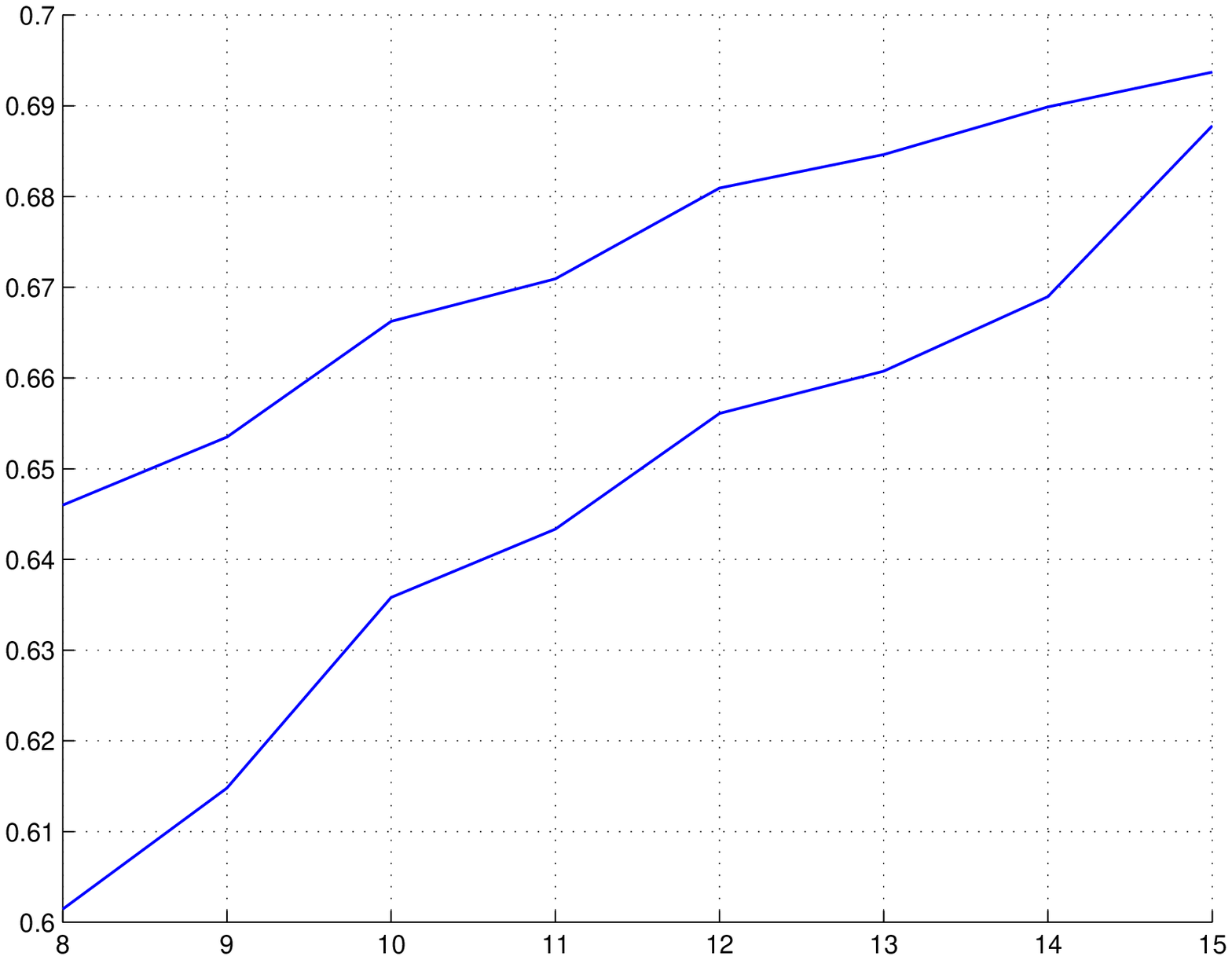}}
  \caption{Relative error of the least-squares reconstruction $\norm{\West-\W}_F/\norm{\W}_F$ for the
    systems in Figure~\ref{fig:acqsys} at different truncation order $K$. The curves from bottom to
    top correspond to respectively: (a) percentage of noise $\sigma_0=20\%, 40\%, 60\%, 80\%$ and
    $100\%$, (b): $\sigma_0=2.5\%$ and $5\%$.}
  \label{fig:resolv_ord}
\end{figure}

\subsection{Shape identification with full aperture of view}
\label{sec:shape-ident-with-full-view} Here we present results of
shape identification obtained using the full aperture of view
setting Figure~\ref{fig:acqsys} (a). For each shape $B_n$ of the
dictionary, we take the three steps mentioned in the beginning of
Section~\ref{sec:numer-exper} on a target $D$ obtained by
transform $B_n$ with the parameters $z=[-0.5,0.5]^\top, s=1.5,
\theta=\pi/3$. The order $K$ is set to 30.

The computation of the error $\ve(D,B_n)$ is represented in Figure~\ref{fig:ident_full_view} by error bars, where the
$m$-th error bar in the $n$-th group corresponds to the error $\ve(D,B_m)$ of the identification
experiment using the generating shape $B_n$. The error bars are arranged in the same way as in 
Figure \ref{fig:8shapes}. The shortest bar in each group is the identified shape
and is marked in green, while the true shape is marked in red in case that the identification
fails. It can be seen that the identification succeeded for all shapes with $20\%$ of noise, and it
failed only for the rectangle with $40\%$ of noise. The estimated scaling factor computed as in (\ref{eq:err_DB}) is close to the
true value $1.5$. The discrepancies between them  are displayed in Figure~\ref{fig:sclest_full_view}.

\def\figwidth{7.5cm}
\begin{figure}[htp]
  \centering
  \subfigure[$\sigma_0=20\%$]{\includegraphics[width=\figwidth]{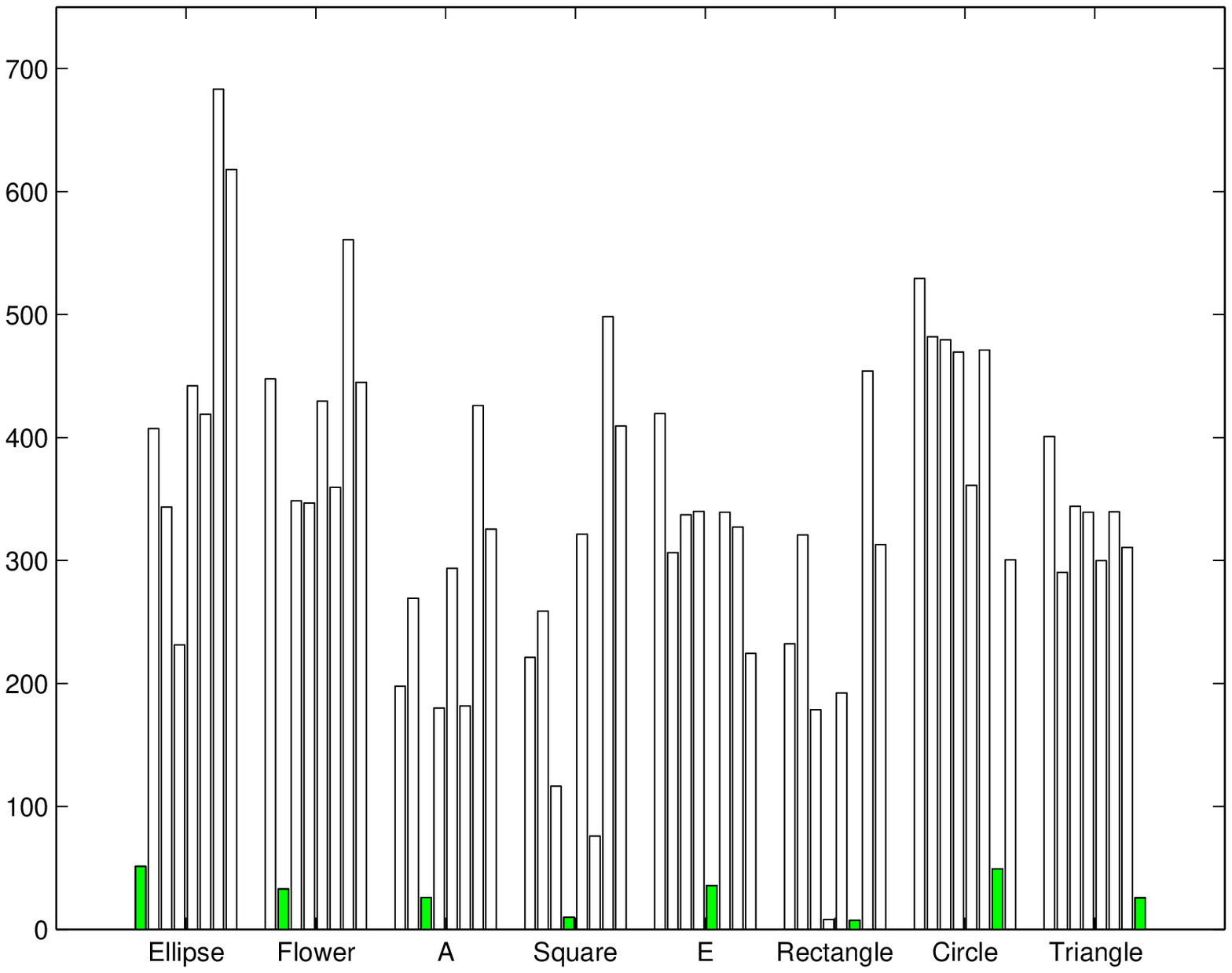}}
  \subfigure[$\sigma_0=40\%$]{\includegraphics[width=\figwidth]{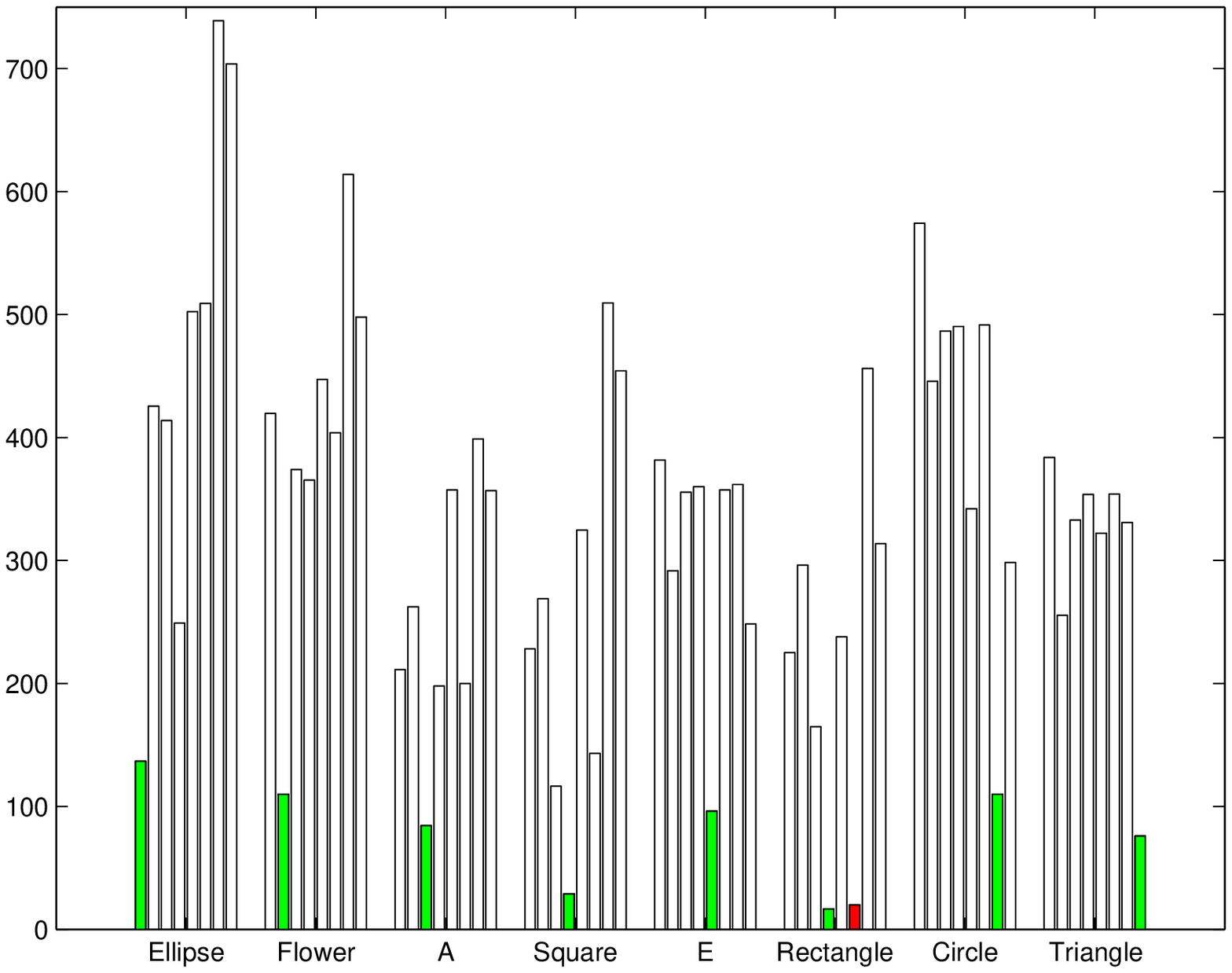}}
  \caption{Results of identification $\ve(D,B_n)$ for all shapes of the dictionary using the full
    aperture of view setting with (a) $\sigma_0=20\%$ and (b) $\sigma_0=40\%$ of noise. All shapes
    were correctly identified in (a), and the rectangle was identified as the square in (b).}
  \label{fig:ident_full_view}
\end{figure}

\def\figwidth{6.5cm}
\begin{figure}[htp]
  \centering
  \subfigure[$\sigma_0=20\%$]{\includegraphics[width=\figwidth]{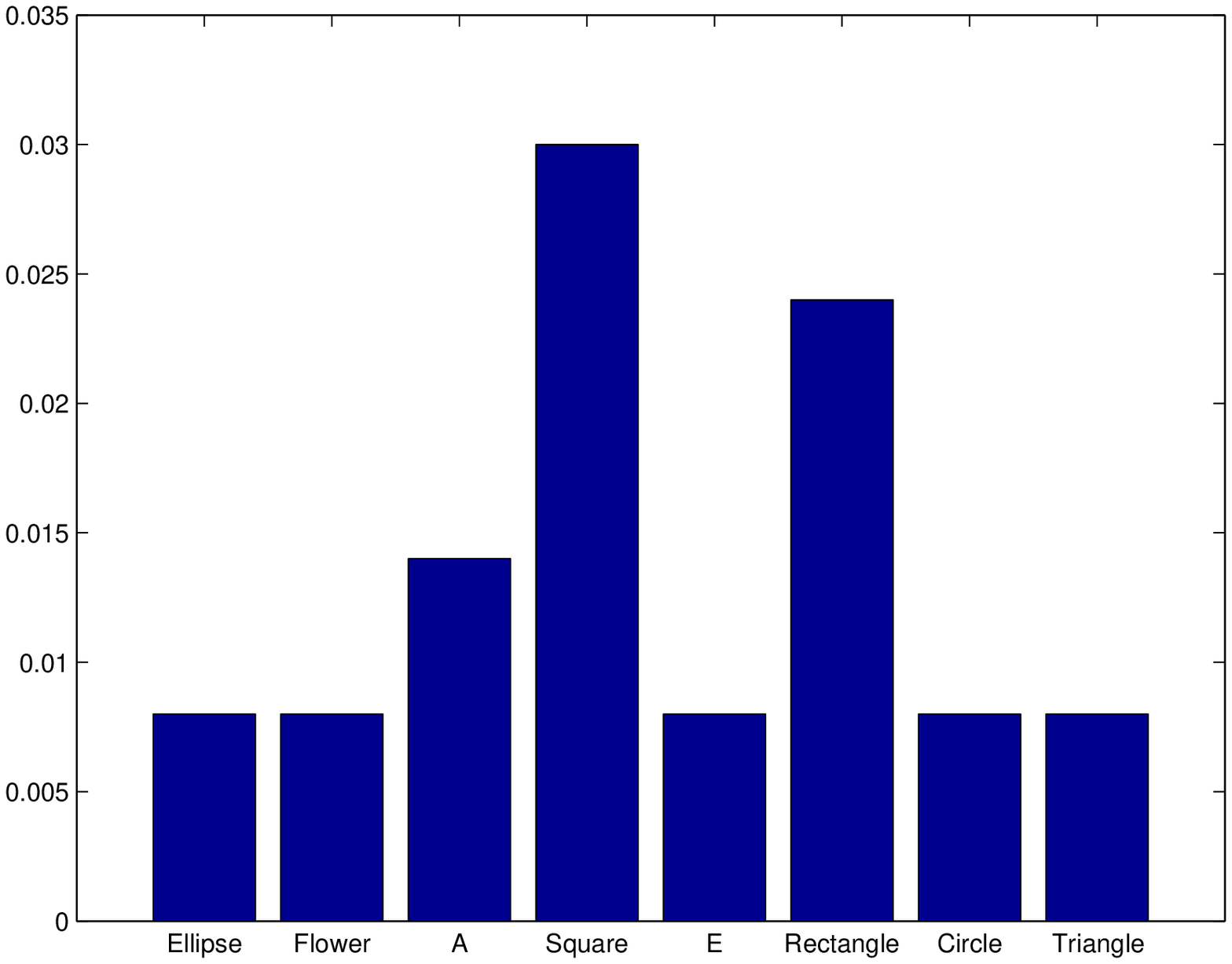}}
  \subfigure[$\sigma_0=40\%$]{\includegraphics[width=\figwidth]{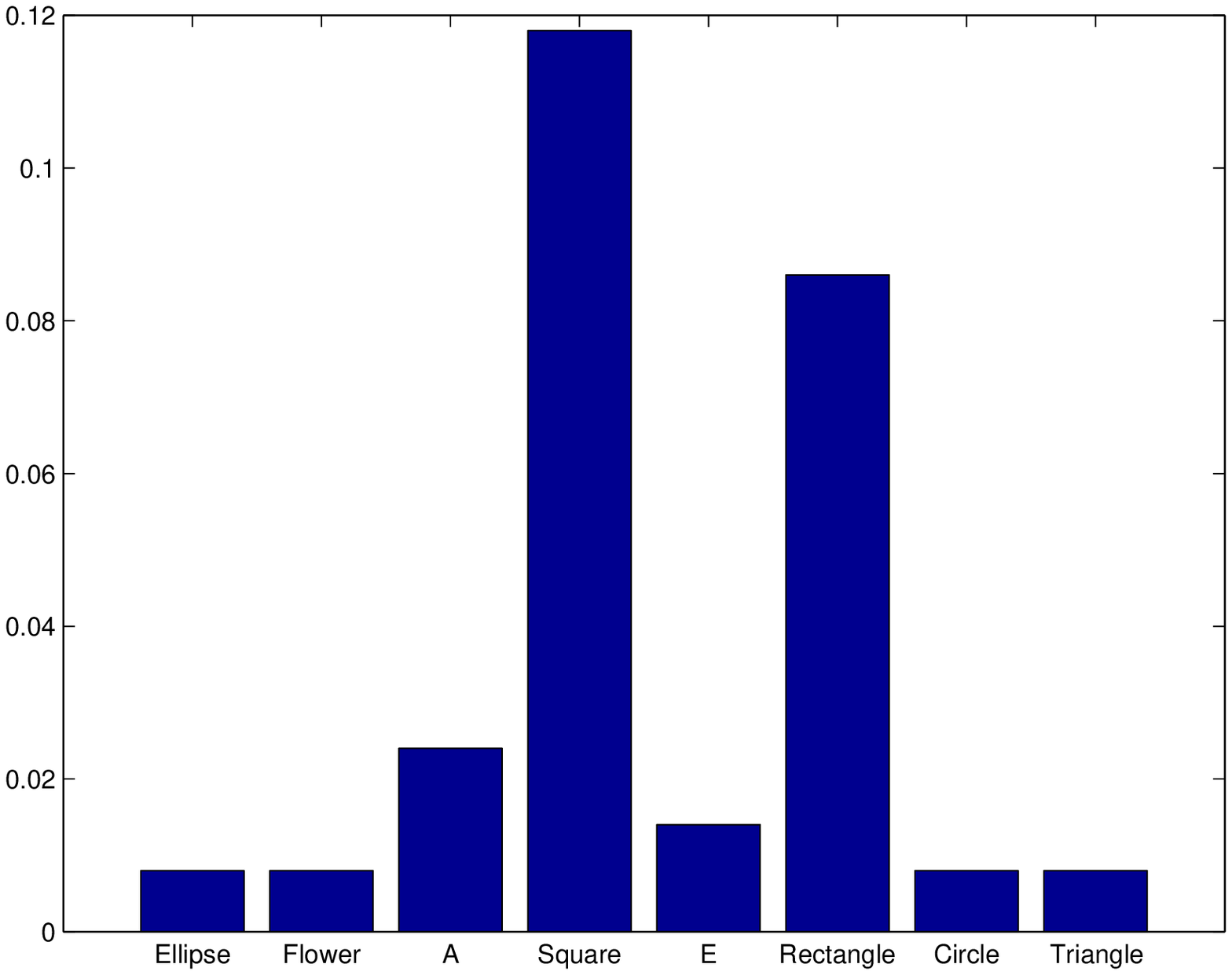}}
  \caption{Difference between the estimated scaling factor and the true one ($s=1.5$) for the
    identified shapes in Figure~\ref{fig:ident_full_view}.}
  \label{fig:sclest_full_view}
\end{figure}


\subsection{Shape identification with limited aperture of view}
\label{sec:shape-ident-with-lim-view}
We use a limited view system as Figure~\ref{fig:acqsys} (b) of small aperture angle $\alpha$, with 512
source/receiver groups (the number of sources/receivers in each group is about $512\times
\alpha/2\pi$) uniformly distributed on a circle of radius $R=10$ around the target. Visually
speaking, the modulus of the measurement $\abs{V_{sr}}$ is similar to that shown in
Figure~\ref{fig:ffpattern_Vsr} with all entries set to zero except the band diagonal (in red).

As analyzed in Section~\ref{sec:exp-parm-settings}, in this case one can not expect to reconstruct
the scattering coefficients with high precision, and consequently compute neither the full far-field
pattern via \eqref{eq:Fourier_SCT} nor the shape descriptor \eqref{eq:Invar_S}. However, thanks to
the relation \eqref{eq:Vsr_far_field} the measurement $\abs{V_{sr}}$ gives directly the band
diagonal part of the far-field pattern $\abs{\Af_D}$, and one can compute the shape descriptors from
the partial far-field pattern hence apply the identification algorithm as described in
Section~\ref{sec:shape-descr-partial}.

The results of identification with the limited aperture are shown in Figure~\ref{fig:ident_lim_view},
where all shapes were identified correctly with $\alpha=\pi/3$ and only the square was missed with
$\alpha=\pi/6$. The performance may deteriorate when the dictionary contains more shapes. In that
case one should increase the aperture angle and broaden the range of operating frequency $[\omin,
\omax]$ in order to compensate the lose of information in the incomplete far-field pattern.

\def\figwidth{7.5cm}
\begin{figure}[htp]
  \centering
  \subfigure[$\alpha=\pi/6$]{\includegraphics[width=\figwidth]{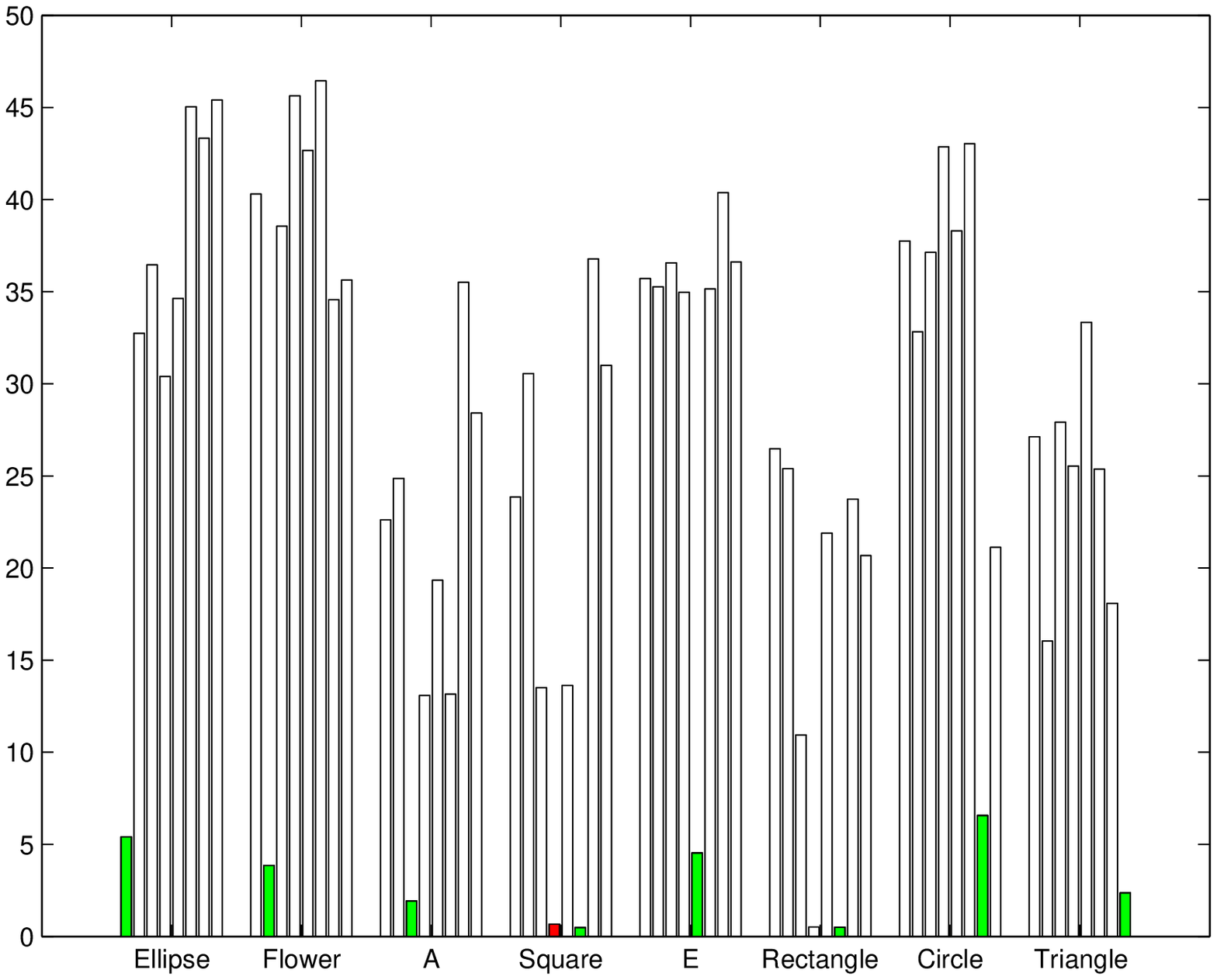}}
  \subfigure[$\alpha=\pi/3$]{\includegraphics[width=\figwidth]{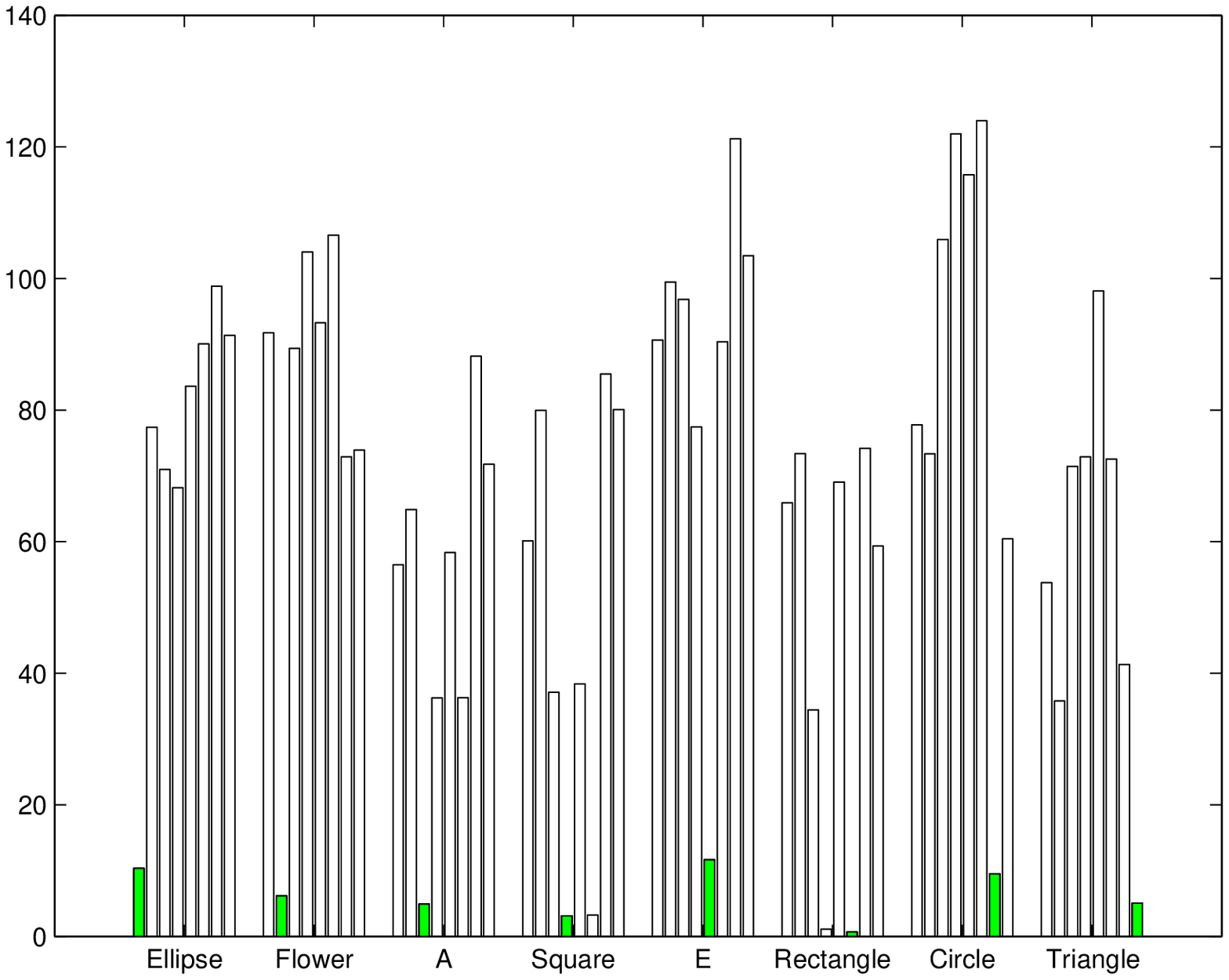}}
  \caption{Results of identification in the limited angle of view setting with angular aperture
    (a) $\alpha=\pi/6$ and (b) $\alpha=\pi/3$. $\sigma_0=20\%$ of noise is added in both cases.}
  \label{fig:ident_lim_view}
\end{figure}

\section{Conclusion}
\label{sec:conclusion}

In this paper we presented a framework of shape identification
using the frequency-dependent dictionary of shape descriptors,
which is invariant to rigid transforms and allows to handle the
scaling within certain ranges. The shape descriptor is based on
the far-field pattern which can be either computed from the
scattering coefficients or read off directly from the measurement.
We analyzed the stability of the reconstruction and presented
results of identification with both full and limited aperture of
view. Our approach in this paper can be used for  tracking the location and the orientation of a target from acoustic echoes. The generalization of our work on target tracking in electrosensing \cite{ammari_tracking_2012} to echolocation will be the subject of a forthcoming paper.

\appendix
\input{appendix_2}

\bibliographystyle{plain}
\bibliography{Biblio}
\end{document}

%% file: appendix_2.tex
\section{Proof of \eqref{eq:xpx_sum}}
\label{sec:proof-of-xpx-sum}

\begin{proof}  
  Let $f(t):=\Paren{\frac c t}^t$ be defined on $\R^+$. Remark that $f$ is decreasing for $t\geq
  c/e$. Therefore, if $k\in\N, k> c/e$, the sum can be bounded by
  \begin{align*}
    \sum_{m>k} \Paren{\frac c m}^m \leq \int_k^{+\infty} f(t) dt = I.
  \end{align*}
  On the other hand, for any $x>1$, we define the function $g(t):=f(t)x^t$ which attains at
  $t=cx/e$ its maximum  $e^{x(c/e)} = a^x$. Then,
  \begin{align*}
    I = \int_k^{+\infty} g(t) x^{-t} dt \leq a^x \int_k^{+\infty} x^{-t} dt = a^x \Paren{\frac{x^{-k}}{\ln x}}.
  \end{align*}
  Since $x>1$ is arbitrary, we deduce
  \begin{align*}
    I \leq \inf_{x>1} a^x \Paren{\frac{x^{-k}}{\ln x}}  \leq \inf_{x\geq ke/c} a^x \Paren{\frac{x^{-k}}{\ln
        x}} \leq \Paren{\inf_{x\geq ke/c} a^x x^{-k}} \Paren{\frac{1}{\ln(ke/c)}}.
  \end{align*}
  The function $a^x x^{-k}$ is convex on $x>1$ and attains  at $x=ke/c>1$ its minimum $(c/k)^k$. Putting everything together, we finally obtain
  \begin{align*}
    I \leq \Paren{\frac c k}^k  \Paren{\frac{1}{1+\ln(k/c)}}
  \end{align*}
  as desired. 
\end{proof}
